\documentclass[11pt]{article}

\makeatletter
\DeclareRobustCommand{\qed}{%
  \ifmmode 
  \else \leavevmode\unskip\penalty9999 \hbox{}\nobreak\hfill
  \fi
  \quad\hbox{\qedsymbol}}
\newcommand{\openbox}{\leavevmode
  \hbox to.77778em{%
  \hfil\vrule
  \vbox to.675em{\hrule width.6em\vfil\hrule}%
  \vrule\hfil}}
\newcommand{\qedsymbol}{\openbox}
\newenvironment{proof}[1][\proofname]{\par
  \normalfont
  \topsep6\p@\@plus6\p@ \trivlist
  \item[\hskip\labelsep\itshape
    #1.]\ignorespaces
}{%
  \qed\endtrivlist
}
\newcommand{\proofname}{Proof}
\makeatother

\usepackage[utf8]{inputenc}

\usepackage{color}
\usepackage{latexsym}
\usepackage{dsfont}
\usepackage{amssymb}
\usepackage{comment}
\usepackage{graphicx}
\usepackage{amsmath,amsfonts,amssymb,theorem,euscript,array,enumerate,amsfonts,mathrsfs}
\usepackage{hyperref}
\usepackage{appendix}
\usepackage[T1]{fontenc}
\usepackage{babel}
\numberwithin{equation}{section}
\usepackage{bbm}
\usepackage{subfigure}
\usepackage{color}

\usepackage{stmaryrd}

\usepackage[algo2e,ruled,vlined]{algorithm2e} 
 \usepackage{algorithm}
 \usepackage{algorithmicx}
\usepackage{algpseudocode}
\usepackage{ marvosym }
\usepackage{csquotes}
\usepackage{hyperref}

\usepackage{mathtools}
\mathtoolsset{showonlyrefs}

\usepackage{comment}

\usepackage{tikz}
\usetikzlibrary{fit,matrix,chains,positioning,decorations.pathreplacing,arrows}

\usepackage{mathtools}
\mathtoolsset{showonlyrefs}

\usepackage{algpseudocode}
\usepackage{algorithm}

\usepackage[normalem]{ulem}

\usepackage[a4paper,left=2.35cm,right=2.35cm, top=3.5cm,bottom=3.5cm]{geometry}
 
\def \trans{^{\scriptscriptstyle{\intercal}}}

\def \Inf{\displaystyle\inf}

\def \b1{\bf{1}}

\def \N{\mathbb{N}}
\def \R{\mathbb{R}}

\def \E{\mathbb{E}}
\def \F{\mathbb{F}}

\def \P{\mathbb{P}}

\def \S{\mathbb{S}}
\def \W{\mathbb{W}}

\def \d{\mathrm{d}}

\def\esssup_#1{\underset{#1}{\mathrm{ess\,sup\, }}}

\def\argmin_#1{\underset{#1}{\mathrm{argmin\, }}}
\def\argmax_#1{\underset{#1}{\mathrm{argmax\, }}}

\def \Ac{{\cal A}}

\def \Fc{{\cal F}}

\def \Ic{{\cal I}}

\def \Lc{{\cal L}}
\def \Pc{{\cal P}}

\def \Nc{{\cal N}}

\def \Vc{{\cal V}}

\def \d{\mathrm{d}}

\def\beqs{\begin{eqnarray*}}
\def\enqs{\end{eqnarray*}}
\def\beq{\begin{eqnarray}}
\def\enq{\end{eqnarray}}

\newtheorem{Theorem}{Theorem}[section]

\newtheorem{Proposition}{Proposition}[section]

\newtheorem{Remark}{Remark}[section]

\numberwithin{equation}{section}

\begin{document}
\title{Linear-quadratic optimal control 
for 
non-exchangeable mean-field SDEs and applications to systemic risk 
}

\author{Anna De Crescenzo%
\thanks{LPSM, Universit\'e Paris Cit\'e and Sorbonne University, Email: decrescenzo at lpsm.paris}
\quad Filippo de Feo%
\thanks{Institut für Mathematik, Technische Universität Berlin, Berlin, Germany and Department of Economics and Finance,  Luiss Guido Carli University, Rome, Italy, Email: defeo@math.tu-berlin.de}
\quad Huyên Pham%
\thanks{Ecole Polytechnique, CMAP, Email:  huyen.pham at polytechnique.edu; This author  is supported by  the BNP-PAR Chair ``Futures of Quantitative Finance", the Chair ``Risques Financiers", and by FiME, Laboratoire de Finance des March\'es de l'Energie, and the ``Finance and Sustainable Development'' EDF - CACIB Chair.}
}

\maketitle

\begin{abstract}
We study the linear-quadratic control problem for a class of non-exchangeable mean-field systems,  which model large populations of heterogeneous interacting agents.  
We explicitly characterize the optimal control in terms of a new infinite-dimensional system of Riccati equations, for which we establish existence and uniqueness. 
To illustrate our results, we apply this framework to a systemic risk model involving heterogeneous banks, demonstrating the impact of agent heterogeneity on optimal risk mitigation strategies.
\end{abstract}

\vspace{5mm}


\vspace{5mm}



\section{Introduction}\label{sec:formulation}

The optimal control of McKean-Vlasov equations, a.k.a. mean-field control problems, 
has been extensively studied over the past decade, leading to significant advancements in stochastic analysis and partial differential equations on the Wasserstein space. These problems have found various applications, particularly in modeling cooperative equilibria in large populations of interacting agents. For a comprehensive treatment of the mathematical foundations and key developments in this field, we refer to the seminal books \cite{cardel18,benetal13}. 

A particularly important subclass of these problems is the linear-quadratic mean-field (LQ-MF) control framework, which has been widely studied due to its analytical tractability and broad range of applications \cite{yong13,huangetal15,sunyong17,baseipham2019}. In the classical setting, interactions among agents are assumed to be of symmetric mean-field type, meaning that agents are homogeneous (or exchangeable). As the number of agents tends to infinity, the problem is typically formulated in terms of a single representative agent whose controlled dynamics depends on its own state distribution.

Recently, there has been growing interest in models incorporating heterogeneous and asymmetric interactions, as seen in complex networks. The theory of graphons \cite{lovasz} has provided a natural framework for modeling such systems and its applications to mean-field SDEs (\cite{bayetal23, copdecpha24, jabin25,cructangpi}), to differential games (\cite{auretal22}), to mean field games (\cite{caihua20, lacker_label-state_2023,LQMFG2024}). 
We also refer to \cite{fenghuhuang2023} for mean-field teams with heterogeneity and quasi-exchangeability. 
More generally, the work in \cite{MFCnon-exch} introduced a framework for the control of non-exchangeable mean-field systems, where agent heterogeneity leads to an infinite-dimensional system of mean-field equations indexed by agent labels. {One of the main difficulties in treating heterogeneous dynamics is related to the fact that }the dynamics of each agent depends not only on its own state distribution but also on the entire collection of labeled state distributions. {In particular, this leads to work on a new space of collection of measures and to develop the dynamic programming approach in this new probabilistic setting.}

\paragraph{Contributions.}  
In this work, we introduce the linear-quadratic non-exchangeable mean-field (LQ-NEMF) control problem, which generalizes the classical LQ-MF framework by incorporating heterogeneous interactions.
In the context of cooperative non-exchangeable games with a finite number $N$ of players, for $i\in\llbracket1,N\rrbracket$, consider the controlled dynamics 
\begin{align}
    \d X^{i,N}_s =& \Big[\beta^i + A^iX^{i,N}_s + \frac{1}{N_i^A}\sum_{j=1}^NG_A(u_i,u_j)X^{j,N}_s+B^i\alpha^{i,N}_s\Big]\d s\\
    &\quad+ \Big[\gamma^i + C^iX^{i,N}_s + \frac{1}{N_i^C}\sum_{j=1}^NG_C(u_i,u_j)X^{j,N}_s+ D^i\alpha^{i,N}_s\Big]\d W^i_s,
\end{align}
where $u_i=i/N$, $G_A,G_C$ are {functions in $L^2([0,1]\times[0,1];\R^{d\times d})$}, measuring the weight of interaction between the agents, {$N_i^Z = \sum_{j=1}^NG_Z(u_i,u_j)$} is the degree of interaction {for $Z=A,C$}, $(W^i)_{i\in\llbracket1,N\rrbracket}$ is a collection of i.i.d. Brownian motions. The aim is to minimize over the controls $(\alpha^{1,N},\dots,\alpha^{N,N})$ a common cost given by
\begin{align}
    &\frac{1}{N}\sum_{i=1}^N\E\Big[  \int_0^T \langle X^{i,N}_s,Q^i X^{i,N}_s \rangle + \langle X^{i,N}_s,\frac{1}{N_i^Q}\sum_{j=1}^N G_{Q}(u_i,u_j)X^{j,N}_s\rangle + \langle \alpha^{i,N}_s,R^i \alpha^{i,N}_s \rangle\\
    &\qquad \qquad + \;  \langle X^{i,N}_T, H^i X^{i,N}_T \rangle + \langle X^{i,N}_T, \frac{1}{N_i^H}\sum_{j=1}^N G_{H}(u_i,u_j)X^{j,N}_T \rangle \Big],
\end{align}
with $G_Q,G_H$ {functions in $L^2([0,1]\times[0,1];\R^{d\times d})$, and {$N_i^Z = \sum_{j=1}^NG_Z(u_i,u_j)$}, {for $Z=Q,H$}}.
When the number of agents $N$ tends to infinity, we formally expect the system to converge to a continuum of players indexed by $u\in U=[0,1]$ (see for instance \cite{bayetal23,copdecpha24}).
Hence, we are led to consider the following limit problem, with state equations
\begin{align}
    &\d X^u_s =\Big[ \beta^u + A^u X^u_s +  \int_U G_{  A}(u,v)\E[X^v_s] \d v + B^u \alpha^u_s \Big] \d s \\
    &\qquad\qquad +\Big[\gamma^u + C^u  X^u_s +    \int_U G_{C}(u,v) \E[ X^v_s]\d v + D^u \alpha^u_s \Big] \d W^u_s,\quad s\in[0,T].
\end{align}
The goal is to minimize, over all admissible controls $(\alpha^u)_u$, a cost functional of the form
\begin{equation}
\begin{aligned}\label{eq:functional_intro}
    &\int_U\E\Big[  \int_0^T \langle X^u_s,Q^u X^u_s \rangle + \langle \E[X^u_s],\int_U G_{Q}(u,v)\E[X^v_s]\d v\rangle + \langle \alpha^u_s,R^u \alpha^u_s \rangle\\
    &\qquad \qquad + \;  \langle X^u_T, H^u X^u_T \rangle + \langle \E[X^u_T], \int_U G_{H}(u,v)\E[X^v_T]\d v \rangle \Big]\d u.
\end{aligned}
\end{equation}
Here the heterogeneous interaction between players is modeled by the coefficients {$G_A,G_C,$ $G_Q,G_H \in L^2(U\times U;\R^{d \times d})$}. This setting includes graphon interaction, but we do not assume these coefficients to be symmetric in $(u,v)$. We prove well-posedness of the state equation and the cost functional (Theorem \ref{exuniqsol}).
{In line with the results obtained in \cite{MFCnon-exch} for more general control problems, also here the dynamics of each player depends on the whole collection of expectations $\{\E[X^u]\}_{u\in U}$, as opposed to the standard LQ MKV control (see for example \cite{baseipham2019}).}

Our goal is to solve the control problem by providing an explicit charac\-terization of the optimal control. Inspired by the literature on LQ problems (see for example for finite dimensional spaces \cite{yong-zhou,peng1992}, and for infinite dimensional spaces \cite{tessitore1992, guattes2005,hutang2022,abimillerpham2021}), we prove a fundamental decomposition of the cost functional (Proposition \ref{prop:fundamental_rel}), which expresses it as the sum of a quadratic term in the control and a constant term. 
This decomposition relies on a new infinite-dimensional system of two Riccati equations, i.e. \eqref{eq:riccati_Ku}, \eqref{eq:abstract_riccati_barK}, for which we establish existence and uniqueness. Equation \eqref{eq:riccati_Ku} is an infinite system on the space $L^\infty(U;\R^{d\times d})$ of decoupled standard Riccati equations and can be solved using standard results. On the other hand, equation \eqref{eq:abstract_riccati_barK} is a new type of Riccati equation on the space $L^2(U\times U;\R^{d\times d})$, which is intrinsically infinite-dimensional (as it encodes heterogenous interactions) and it is not covered by the available literature. {This marks a fundamental difference with the standard LQ MKV control, where the two Riccati equations characterizing the problem are standard matrix Riccati equations that can be solved using the standard theory presented in \cite{yong-zhou}.}
The fundamental decomposition provided in Proposition \ref{prop:fundamental_rel} plays a crucial role in solving the Riccati equation \eqref{eq:abstract_riccati_barK}, by finding a suitable a-priori estimate on its solution (Proposition \ref{cor:a_priori_estimate_barK}), and in deriving the optimal control (Theorem \ref{th:verification}).
These results extend previous findings in the LQ-MF setting.

To illustrate the practical relevance of our results, in Section \ref{sec:systemic_risk} we introduce a systemic risk model with a continuum of heterogeneous banks, indexed by $u$, extending the one  in \cite{systemicrisk} with homogeneous interactions (i.e. a standard LQ-MF problem). Indeed, we assume that the log-monetary reserve of each bank follows a dynamics of the form
\begin{align} 
\d X^u_s &= \; \Big[
k\big(
X^u_s -\int_U \tilde G_{k}(u,v)\,\mathbb E[ X^{v}_s] \,\d v \big)
+\alpha^u_s \Big] \,\d s + \sigma^u\, \d W^u_s,
\end{align} 
where each bank can control its rate of borrowing/lending to a central bank via the policy $\alpha^u$ and $k \le 0$ is the rate of mean-reversion in the interaction from borrowing and lending between the banks. The central bank aims to mitigate  systemic risk by minimizing an opportune aggregate cost functional, leading to a non-exchangeable mean field control problem of the form \eqref{eq:functional_intro}.
Applying the theory developed, we  solve the control problem.

\paragraph{Outline of the paper.}
The remainder of this paper is organized as follows. In Section~\ref{sec:LQNEMF}, we formally introduce the LQ-NEMF control problem. Section~\ref{sec:derivation} presents the infinite-dimensional Riccati system associated with LQ-NEMF and establishes the fundamental relation. The existence and uniqueness of a solution to the Riccati system are proved in Section~\ref{sec:existence}, while Section~\ref{sec:verif_thm} details  the derivation of the optimal control. In Section~\ref{sec:systemic_risk}, we apply our results to the systemic risk model, demonstrating the impact of agent heterogeneity in financial networks. In Appendix \ref{app:exuniqsol} we prove well-posedness of the state equation and of the cost functional. In Appendix \ref{app:different-formulations} we provide different formulations of our control problem.

\paragraph{Notations.} In the following, for each $q\in\N$, $\S^q$ (resp. $\S^q_+$, $\S^q_{>+}$) is the set of $q$-dimensional symmetric (resp. positive semi-definite, positive definite) matrices. Given a vector $x\in\R^d$, we denote by $x^i$ the $i$-th component of vector $x$. We denote by $M\trans$ the transpose of a matrix $M$, and 
for two vectors $x,y$, we denote by $<x,y>$ $=$ $x\trans y$ its Euclidian scalar product.
We denote by $C^d_{[t,T]}$ the space of continuous functions $f:[t,T]\to \R^d$.
Given a Banach space $X$, we denote by $\Lc(X)$ the space of linear bounded operators $L$ on $X$, with norm defined by $\|L\|_{\Lc} = \sup_{\|x\|_X\ne 0}\frac{\|Lx\|_X}{\|x\|_X}$.

Let $T>0$ be a finite horizon, $(\Omega,\Fc,\P)$ a complete probability space, $U:=[0,1]$ endowed with its Borel $\sigma$-algebra. 
For every $u\in U$, we consider an $\R$-valued standard Brownian motion $W^u$ and an independent real random variable $Z^u$ having uniform distribution in $(0,1)$. For simplicity we consider dynamics driven by real-valued Brownian motions, but all the results can be extended to the case of 
multi-dimensional Brownian motions. We assume that $\{(W^u,Z^u)\,:\, u\in U\}$ is an independent family. 
For every $u\in U$
we denote by $(\Fc^{W^u}_t)_{t\ge 0}$ the natural filtration generated by $W^u$, by $\sigma(Z^u)$ the $\sigma$-algebra generated by $Z^u$ 
and by  
$\F^u=(\Fc^{u}_t)_{t\ge 0}$ the filtration given by 
$\Fc^{u}_t= \Fc^{W^u}_t\vee \sigma(Z^u) \vee \Nc, t\ge0,
$
where $\Nc$ is the family of $\P$-null sets. 
For a Banach space $E$, with norm $|.|_{_E}$, we denote by $L^2(U;E)$, resp. $L^\infty(U;E)$,  the set of elements $\phi$ $=$ 
$(\phi^u)_{u\in U}$, s.t. $u$ $\mapsto$ $\phi^u$ $\in$ $E$  is measurable, and 
$\int_U |\phi^u|_{_E}^2 \d u$ $<$ $\infty$, resp. $\sup_{u\in U} |\phi^u|_{_E}$ $<$ $\infty$.  
{We also denote by $L^2(U\times U;E)$, resp. $L^\infty(U\times U;E)$ the set of measurable functions $G$ $:$ $U\times U$ $\rightarrow$ $E$, s.t. $|G|^2_{_{L^2_{U\times U}}}$ $:=$ $\int_U\int_U |G(u,v)|_E^2\, \d u \d v$ $<$ $\infty$, resp. $\sup_{u,v \in U\times U} |G(u,v)|_{_E}$ $<$ $\infty$.} 
{For a kernel $G\in L^2(U \times U;\mathbb R^{d \times d})$, we will  use Hilbert Schimdt integral operators, denoted by $\mathbf G$ , i.e. 
\begin{equation}\label{eq:Hilbert-schmidt_integral}
    \mathbf G\in \mathcal L(L^2(U \times U;\mathbb R^{d })), \quad  \mathbf G f:= \int_U G(u,v) f^v\d v, \quad \|\mathbf G\|_{\mathcal L} \leq |G|_{_{L^2_{U\times U}}},
\end{equation}
We remark that we will use this notation throughout the paper.}

{We also introduce the notation 
\begin{equation}\label{eq:GS}
    G^S(u,v) = \frac{G(u,v)+G(v,u)\trans}{2}
\end{equation}
and we observe that 
\begin{align}\label{eq:symmetric_structure_graphon_cost_term_rem}
    \int_U \langle y^u,\int_U  G(u,v)y^v\d v \rangle \d u =  \int_U \langle y^u,\int_U G^S(u,v)y^v\d v \rangle \d u , \quad \forall y \in L^2(U;\mathbb R^d),
\end{align}
Indeed, we have
\begin{align}
& \int_U \langle y^u, \int_U G(u,v) y^v\d v \rangle \d u=\int_U\int_U \langle y^v,G(v,u)y^u\rangle\d v\d u = \int_U\int_U \langle y^u, G(v,u)\trans y^v\rangle\d v\d u,
\end{align}
so that, summing $\int_U \langle y^u, \int_U G(u,v)y^v\d v \rangle \d u $ on both members of the previous equality and dividing by $2$, we have \eqref{eq:symmetric_structure_graphon_cost_term_rem}.
Observe that $G^S(u,v) = G^S(v,u)\trans$.
}

In the following, we write $M>0$ for a constant that might change from line to line. 

\section{Problem formulation} \label{sec:LQNEMF}

\paragraph{State equation.}
We consider a collection of state $X$ $=$ $(X^u)_u$ process, indexed by $u\in U=[0,1]$, and controlled by the collection 
$\alpha$ $=$ $(\alpha^u)_u$, driven by  
\begin{align}\label{dynamics}
\begin{cases}
    &\d X^u_s =\left [ \beta^u + A^u X^u_s +  \int_U G_{  A}(u,v)\bar X^v_s \d v + B^u \alpha^u_s \right] \d s \\
    &\qquad\qquad +\left [\gamma^u + C^u  X^u_s +    \int_U G_{C}(u,v) \bar X^v_s\d v + D^u \alpha^u_s \right ] \d W^u_s,\quad s\in[t,T]\\ 
    & X_t^u=\xi^u, \qquad u \in U, 
\end{cases}
\end{align}
where 
 {$\beta \in L^2(U;\R^d)$, $\gamma\in L^2(U;\R^{d})$,} $A \in L^\infty(U;\R^{d\times d}),$ $C\in L^\infty(U;\R^{d\times d})$, $B\in L^\infty(U;\R^{d\times m}),$ $D\in L^\infty(U;\R^{d\times m})$, {$G_{A}\in L^2(U\times U; \R^{d\times d})$, $G_{C}\in L^2(U\times U; \R^{d\times d})$}, and we have used the notation $$\bar X_s^u:=\mathbb E[X_s^u].$$  
 
The set $\Ic=\Ic_t$ of initial conditions contains collections ${\xi}=(\xi^u)_u$ of $\R^d$-valued random {variables} such that $\xi^u$ is $\Fc_t^u$-measurable for each $u\in U$, the maps $u\mapsto \E[\xi^u]$, $u\mapsto\E[\xi^{u,i}_{s}\xi^{u,j}_{s}]$ are Borel measurable for all 
$s\in[0,T], i,j\in{1,\dots,d}$, and $\int_U\E[|\xi^u|^2]\d u <\infty$.

Given $\xi\in\Ic$, the set $\Ac=\Ac_t(\xi)$ of admissible controls  starting from time $t$,  
is  defined as follows.
For an arbitrary Borel measurable function 
$\alpha: U\times [t,T]\times C([t,T];\R)\times (0,1)\to   \mathbb R^m$, 
we define 
$
\alpha_s^u=\alpha(u,s, W^u_{\cdot\wedge s},Z^u),$ $s\in [t,T],\,u\in U,
$
where $W^u_{\cdot\wedge s}$ is the path $r\mapsto 
W^u_{r\wedge s}$, $s\in [t,T]$.
We note that each process $(\alpha_s^u)_s$ is $\F^u$-predictable.
We say that $\alpha$ is an 
 admissible control policy (or simply a policy) if $u\mapsto\E[\alpha^u\xi^u]$ is Borel measurable and
$\int_U \int_t^T\E[|\alpha_s^u|^2]\,\d s\,\d u\; < \; \infty.$
Observe that this integral is well defined, since $\E[|\alpha^u_s|^2]$ is measurable in $u$. Indeed, we can write 
$\E[|\alpha^u_s|^2] = \E[|\alpha(u,s,W^u_{\cdot\wedge s},Z^u)|^2] = \int_{C_{[0,s]\times(0,1)}}|\alpha(u,s,w(\cdot\wedge s),z)|^2\W_T(\d w)\d z,$
where $\W_T$ denotes the Wiener measure on $C_{[0,T]}$.

{\begin{Remark}
    Observe that the fact that the admissible control $\alpha^u$ is measurable with respect to the (idiosyncratic) noise $(W^u,Z^u)$ (and not to the full filtration generated by the collection $(W^u,Z^u)_u$) is in line with standard mean-field models (see for instance the linear quadratic case in \cite{baseipham2019}), where the control of the (representative) player in the mean-field limit is adapted to the filtration generated by the (representative) Brownian motion. 
\end{Remark}}

\vspace{2mm}

We point out that the setting considered here is new as  we do not require the measurability of $u \mapsto \mathbb P_{\xi^u}$ as in  \cite{bayetal23,MFCnon-exch}. Next, we show the well-posedness of the system \eqref{dynamics}.

\begin{Theorem}\label{exuniqsol}
    Let $t\in[0,T]$, $\xi\in\Ic$, and $\alpha\in\Ac$. Then there exists a unique solution $X=(X^u)_u$ to the state equation \eqref{dynamics}, i.e.,  a family $X=(X^u)_u$ of stochastic processes with values in $\R^d$ such that the maps $u\mapsto\E[X^u_s]$, $u\mapsto\E[X^{u,i}_{s}X^{u,j}_{s}]$ are Borel measurable $\forall s,i,j$, each process $X^u$ is continuous and $\F^u$-adapted and $\int_U\E[\sup_{s\in[t,T]}|X^u_s|^2]\d u<\infty$. 
\end{Theorem}
\begin{proof}
See Appendix \ref{app:exuniqsol}
\end{proof}

\begin{Remark}\label{rem:non-autonomous_coeff}
\begin{enumerate}
    \item {The result extends with straightforward modifications} in the proof to the case of time-dependent coefficients if {$$\sup_{t\in[0,T]}\{|\beta_t|_{L^2_U},|\gamma|_{L^2_U},|A_t|_{L^\infty_U},|C_t|_{L^\infty_U}, |B_t|_{L^\infty_U},|D_t|_{L^\infty_U},|G_{A,t}|_{L^2_{U\times U}},|G_{C,t}|_{L^2_{U\times U}} \}<\infty$$}
    \item The measurability of the maps $u\mapsto\E[X^u_s]$, $u\mapsto\E[X^{u,i}_{s}X^{u,j}_{s}]$ ensures (respectively) the well-posedness of the terms of type $\int_UG_L(u,v)\bar X^v_s\d v$ in the state equation (for $G_L = G_A,G_C$) and the terms of type $\int_U\E[\langle X^u_s,L^uX^u_s\rangle]$ in the cost functional (for $L=Q,H$).
\end{enumerate}
\end{Remark}

\vspace{3mm}

By standard estimates and Gronwall lemma, there exists $M_T>0$ such that for all $\xi,\alpha$, 
\begin{small}
\begin{align}\label{estimX}
\int_U \E \bigg[\sup_{s\in [t,T]} |X_s^u|^2\bigg] \,\d u
& \le \;  
M_T\left((T-t){\int_U(|\beta^u|^2 + |\gamma^u|^2) \d u}+ \int_U \E[|\xi^u|^2]\,\d u
+\int_U \int_t^T\E[|\alpha_s^u|^2]\,ds\,\d u\right).
\end{align}
\end{small}

\paragraph{Functional cost.} Given $t\in[0,T]$ and $\xi\in\Ic$, the goal is  to minimize, over all $\alpha \in \mathcal A$, a  functional of the form 
\begin{align}\label{eq:J_theory}
    J_{}(t,\xi,\alpha) &:= \;  \int_U\E\Big[  \int_t^T \langle X^u_s,Q^u X^u_s \rangle + \langle \bar X^u_s,\int_U G_{Q}(u,v)\bar X^v_s\d v\rangle + \langle \alpha^u_s,R^u \alpha^u_s \rangle\\
    &\qquad \qquad + \;  \langle X^u_T, H^u X^u_T \rangle + \langle \bar X^u_T, \int_U G_{H}(u,v)\bar X^v_T\d v \rangle \Big]\d u,
\end{align}
where $Q, H \in L^\infty(U;\S_+^{d}),$ $G_{ Q}, G_{ H}$ $\in$ ${L^2(U\times U;\R^{d\times d})}$, $R\in L^\infty(U;\S_{>+}^{m})$ {and there exists a $c>0$ such that for almost every $u\in U$
\begin{equation}\label{ass:posR}
    \langle R^uy,y\rangle \ge c|y|^2,\quad\forall y\in\R^d.
\end{equation}
}

\vspace{1mm}

{
Throughout the whole paper, we assume
\begin{equation}
\begin{aligned}\label{eq:positivity_hpII}
   & \int_U \mathbb E \langle X^u,Q^uX^u\rangle du + \int_U \langle \bar X^u,  \int_UG_{Q}(u,v)\bar X^v \d v \rangle  du  \geq 0;\\
   & \int_U \mathbb E \langle X^u,H^u  X^u\rangle du + \int_U \langle  \bar X^u,  \int_UG_{H}(u,v)\bar X^v \d v \rangle  du  \geq 0,
\end{aligned}
\end{equation}
for every $X$ such that $\Lc(X)\in L^2(\Pc_2(\R^d))$. 
Under such assumption, it follows
\begin{align}\label{eq:J_geq0}J(t,\xi,\alpha)\geq 0, \quad \forall t,\xi,\alpha.
\end{align}}
{
\begin{Remark}
    Clearly, a sufficient condition which implies \eqref{eq:positivity_hpII} is 
\begin{equation}
\begin{aligned}\label{eq:positivity_hp}
   &  \int_U \langle z^u,  \int_UG_{Q}(u,v)z^v \d v \rangle  du  \geq 0,  \quad \int_U \langle y^u,  \int_UG_{H}(u,v)y^v \d v \rangle  du  \geq 0, \quad \forall z \in L^2(U)
\end{aligned}
\end{equation}
(plus that $Q^u$ and $H^u$ are non-negative).
   Define the linear operator 
   \begin{align}
     & \boldsymbol{G}_Q : L^2(U;\R^d)  \rightarrow L^2(U;\R^d),  \quad
   \boldsymbol{G}_Q(z)(u):= \int_U G_{Q}(u,v)z^v\d v,
\end{align}
for every $z \in L^2(U;\R^d)$; analogously, define the linear operator $ \boldsymbol G_H$ associated to $G_H$.
Then, \eqref{eq:positivity_hp} correspond to the requirement that $\boldsymbol G_Q, \boldsymbol G_H$ are non-negative operators, i.e.
$$\langle  z,\boldsymbol G_Q z \rangle_{L^2_{U\times U}} \geq 0, \quad \langle  z ,\boldsymbol G_H z \rangle_{L^2_{U\times U}} \geq 0.$$
However, \eqref{eq:positivity_hp} does not need to be satisfied in the centered formulation (see subsection \ref{subsec:centered_formulation}) for \eqref{eq:positivity_hpII} to be true. 
\end{Remark}
}

\vspace{2mm}

From \eqref{estimX}, we see that for all $t$ $\in$ $[0,T]$, $\xi$ $\in$ $\Ic$, $\alpha$ $\in$ $\Ac$, 
\begin{small}
\begin{align}\label{estimJ}
|J_{}(t,\xi,\alpha)| \le \;  
M_T\left( (T-t){\int_U(|\beta^u|^2 + |\gamma^u|^2)\d u} +  \int_U \E[|\xi^u|^2]\,\d u
+\int_U \int_t^T\E[|\alpha_s^u|^2]\,\d s\,\d u\right).
\end{align}
\end{small}

\begin{Remark}
    Notice that our formulation is different from the one on Hibert spaces  in \cite{cosetal23} (and for LQ MFGs \cite{salgozghi25,liufir2024}), { where the goal is to build a theory for abstract McKean--Vlasov SDEs under assumptions of homogeneity and anonymity.} In our case we cannot reformulate our dynamics as a {McKean--Vlasov} SDE with values on a Hilbert space.
    {In particular, the standard formulation which uses Wiener processes on Hilbert spaces (see for example \cite{fabbrigozziswiech}) is not well-suited for our model. This is due to the fact that we  consider an uncountable collection of heterogeneous players, thus an uncountable collection of independent Brownian motions, which is a different object than the Wiener process on Hilbert space.}
    Furthermore, {in our setting we do not have} the measurability of the Brownian motions, and therefore of the state, with respect to the variable $u$ (see for example \cite{sunexactlln06}). {A way to retrieve the measurability of the states with respect to the variable $u$ would be to use Fubini extensions (see for example \cite{sunexactlln06,copdecpha24,cructangpi}). However, even in this situation we would not be able to use the standard theory for the above reasons.}
\end{Remark}

\section{Infinite dimensional system and fundamental relation}\label{sec:derivation}
\subsection{Derivation of the Riccati equations}{As in the literature of linear-quadratic control problems, we start with a quadratic ansatz for the value function, i.e. we consider
\begin{align} 
    \mathcal{V}(t,\xi) & := \int_U\E[\langle \xi^u, K^u_t \xi^u \rangle]\d u 
    + \int_U\int_U\langle\bar \xi^u,\bar K_t(u,v)\bar \xi^v\rangle\d v \d u \\
    &\qquad\qquad\qquad\qquad + \; 2 \int_U \E[\langle  Y^u_t,\xi^u\rangle] \d u + \int_U \Lambda^u_t\d u,
\end{align}
where $K\in C([0,T];\S^d),\,\bar K\in C([0,T];L^2(U\times U;\R^{d\times d})),\, Y\in C([0,T];L^2(U;\R^d)),\,\Lambda\in C([0,T];L^\infty(U;\R))$ are suitable functions to be determined later. The terminal condition is given by 
\begin{align}
\mathcal{V}(T,\xi) :=& \int_U\E[\langle \xi^u, H^u \xi^u \rangle]\d u 
    + \int_U\int_U\langle\bar \xi^u,G_H(u,v)\bar \xi^v\rangle\d v \d u\\
    =& \int_U\E[\langle \xi^u, H^u \xi^u \rangle]\d u 
    + \int_U\int_U\langle\bar \xi^u,G^S_H(u,v)\bar \xi^v\rangle\d v \d u,
\end{align}
where in the second equality we have used \eqref{eq:symmetric_structure_graphon_cost_term_rem}.
Then a standard strategy in finite and infinite dimension for LQ problems is to use this ansatz and Ito's formula in order to decompose the cost functional $J$ into the sum of a quadratic term in the control and a constant term, leading to suitable Riccati equations. 
Following this strategy, we introduce:
}
\begin{align}
    \mathcal{V}_t & := \int_U\E[\langle X^u_t, K^u_t X_t^u \rangle]\d u 
    + \int_U\int_U\langle\bar X_t^u,\bar K_t(u,v)\bar X_t^v\rangle\d v \d u  + \; 2 \int_U \E[\langle  Y^u_t,X^u_t\rangle] \d u + \int_U \Lambda^u_t\d u.
\end{align}
For the sake of simplicity, we provide the following calculations in the case of $\beta^u=0$, $\gamma^u=0$ (recall also Remark \ref{solnulle}). Hence, we have
\begin{align}
    \mathcal{V}_t & = \int_U\E[\langle X^u_t, K^u_t X_t^u \rangle]\d u 
    + \int_U\int_U\langle\bar X_t^u,\bar K_t(u,v)\bar X_t^v\rangle\d v \d u .
\end{align}
Due to the definition of $\mathcal{V}_t$, we can use standard It\^o's formula for a.e. $u$ and then integrate over $U.$ Notice that in this way, we do not need the notion of derivatives with respect to measures (e.g. see \cite{MFCnon-exch}).
Proceeding in this way, we compute (denoting $\sigma_t^u:=\gamma^u + C^u  X^u_t +    \int_U G_{C}(u,v) \bar X^v_t\d v + D^u \alpha^u_s$): 
\begin{small}
\begin{align}
    &\d \bigg( \mathcal{V}_t + \int_0^t\int_U\E\bigg[ \langle X^u_s,Q^uX^u_s \rangle + \langle \bar X^u_s,  \int_UG_Q(u,v)\bar X^v_s\d v\rangle + \langle \alpha^u_s,R^u\alpha^u_s \rangle \bigg]\d u\d s \bigg)\\
    &=\int_U \mathbb E \left[\langle \d X_t^u,K^u_t X_t^u\rangle+ \langle X_t^u,\d (K^u_t X_t^u)\rangle  + \langle  {\sigma_t^u, K_t^u \sigma_t^u} \rangle \d t  \right]du\\
    &\quad   + \int_U\int_U \left[\langle \d\bar X_t^u, \bar K_t(u,v) \bar X_t^v\rangle+ \langle \bar X_t^u,d(\bar K_t(u,v) \bar X_t^v)\rangle  \right]du dv\\
    & \quad + \int_U \E\left[\langle X^u_t,Q^u X^u_t\rangle + \int_U\langle \bar X^u_t, G_Q(u,v)\bar X^v_t\rangle\d v + \langle \alpha^u_t, R^u\alpha^u_t\rangle\right]\d u\\
    &=\int_U \mathbb E \left[\langle \d X_t^u,K^u_t X_t^u\rangle+ \langle X_t^u,\d K^u_t X_t^u\rangle +  \langle  X_t^u, K^u_t \d X_t^u\rangle + \langle \sigma_t^u,K_t^u \sigma_t^u \rangle \d t \right]du\\
    &\quad   + \int_U\int_U \left[\langle \d\bar X_t^u, \bar K_t(u,v) \bar X_t^v\rangle+ \langle \bar X_t^u,\d \bar K_t(u,v) \bar X_t^v\rangle +  \langle \bar  X_t^u, \bar K_t(u,v) \d \bar X_t^v\rangle  \right]du dv\\
     & \quad + \int_U \E\left[\langle X^u_t,Q^u X^u_t\rangle + \int_U\langle \bar X^u_t, G_Q(u,v)\bar X^v_t\rangle\d v + \langle \alpha^u_t, R^u\alpha^u_t\rangle \right]\d u\\
     & = \int_U\E\bigg[ 
     \langle A^uX^u_t +  \int_UG_{  A}(u,v)\bar X^v_t\d v +B^u\alpha^u_t,K^u_tX^u_t \rangle+ \langle X^u_t, \dot K^u_tX^u_t\rangle \\
     & \quad 
     + \langle K^u_t A^uX^u_t + K^u_t  \int_UG_{  A}(u,v)\bar X^v_t\d v +K^u_t B^u\alpha^u_t, X^u_t \rangle\\
    &\quad  + \langle K^u_t(C^uX^u_t +   \int_UG_{C}(u,v)\bar X^v_t\d v + D^u\alpha^u_t ), C^uX^u_t +   \int_UG_{C}(u,w)\bar X^w_t\d w + D^u\alpha^u_t \rangle\\
       &\quad  + \int_U \langle  A^u\bar X^u_t +  \int_UG_{  A}(u,w)\bar X^w_t\d w + B^u\E[\alpha^u_t], \bar K_t(u,v)  \bar X^v_t \rangle \d v  \\
    &\quad + \int_U\langle\bar X^u_t, \dot{\bar K}_t(u,v)\bar X^v_t \rangle\d v  + \int_U \langle \bar X^u_t, \bar K_t(u,v)\left( A^v\bar X^v_t +   \int_UG(v,w)\bar X^w_t\d w + B^v\E[\alpha^v_t] \right) \rangle \d v\\
    & \quad + \langle X^u_t,Q^u X^u_t\rangle + \int_U\langle \bar X^u_t, G_Q(u,v)\bar X^v_t\rangle\d v + \langle \alpha^u_t, R^u\alpha^u_t\rangle
    \bigg]\d u \d t.
\end{align}
\end{small}
Next, we want to rewrite the previous expression in a  suitable form, based on the following observations:
\begin{itemize}
    \item via matrix transpositions, we rewrite all terms of the type $\langle (\ldots) X_t^u, (\ldots)  X_t^u\rangle $ in the form $\langle (\ldots)X_t^u,X_t^u \rangle$;
    \item we can write all the terms depending on $\bar X$ in the form $\langle (\dots)\bar X^v_t,\bar X^u_t \rangle$. For example:
    \begin{align}
        &\int_U\int_U\int_U \langle K^u_t  G_{C}(u,v)\bar X^v_t,   G_{C}(u,w)\bar X^w_t\rangle\d w\d v\d u\\
        &= \int_U\int_U\int_U \langle K^w_t  G_{C}(w,v)\bar X^v_t, G_{C}(w,u)\bar X^u_t\rangle\d w\d v\d u\\
        & = \int_U\int_U\int_U\langle G_{C}(w,u)^T K^w_t  G_{C}(w,v)\bar X_t^v, \bar X^u_t\rangle\d w\d v\d u;
    \end{align}
    \item we can keep only linear terms in $\alpha$ (and not $\E[\alpha]$) observing that $\E[\psi_t\trans\E[\alpha^u_t]] = \E[\E[\psi_t\trans]\alpha_t^u]$;
    \item {on the right hand side of the above equation, using \eqref{eq:symmetric_structure_graphon_cost_term_rem}, we substitute $G_Q$ with $G^S_Q$. This step ensures that the abstract Riccati equation that will be derived satisfies a suitable symmetry property.}
\end{itemize}
Proceeding in this way, we rewrite the previous expression as
     \begin{align}
    &\d \bigg( \mathcal{V}_t + \int_0^t\int_U\E\bigg[ \langle X^u_s,Q^uX^u_s \rangle + \langle \bar X^u_s,  \int_UG^S_Q(u,v)\bar X^v_s\d v\rangle + \langle \alpha^u_s,R^u\alpha^u_s \rangle \bigg]\d u\d s \bigg)\\
    & = \int_U\E\bigg[ \langle (\dot K^u_t + (A^u)^TK^u_t + K^u_tA^u + (C^u)^TK^u_tC^u + Q^u)X^u_t,X^u_t\rangle \bigg]\d u\\
    &\quad  + \int_U\int_U\int_U\E\bigg[ \langle(\dot{\bar K}_t(u,v) + G_{  A}(u,v)K^u_t  + G_{  A}(v,u)^TK^v_t + (C^u)^TK^u_t  G_{C}(u,v)\\
    &\quad + (G_{C}(v,u))^TK^v_tC^v + G_{C}(w,u)^TK^w_t  G_{C}(w,v)\\
    &\quad + (A^u)^T\bar K_t(u,v) + \bar K_t(u,v)A^v + \bar K_t(u,w)    G_{  A}(w,v) + (G_{  A}(w,u))^T \bar K_t(w,v)\\
    &\quad + {G^S_Q(u,v)})\bar X^v_t,\bar X^u_t\rangle \bigg]\d w\d v\d u \d t\\
    &\quad + \int_U\E\bigg[ 2\bigg\langle ((B^u)^TK^u_t + (D^u)^TK^u_tC^u)X^u_t + \int_U\big((B^u)^T \bar K_t(u,v)  \\
     &\quad \quad \quad+ (D^u)^TK^u_t  G_{C}(u,v)\big)\bar X^v_t\d v, \alpha^u_t \bigg\rangle \\
     &\quad\quad\quad\quad+ \langle \alpha^u_t, \bigg(R^u + (D^u)^TK^u_tD^u\bigg)\alpha^u_t\rangle \bigg]\ d u \\
    &= \int_U\E\left[\langle (\dot K^u_t + \Phi_t^u)X^u_t,X^u_t\rangle\right]\d u \d t+ \int_U\int_U\E\left[ \langle(\dot{\bar K}_t(u,v) + \Psi_t(u,v))\bar X^v_t,\bar X^u_t\rangle\right]\d v\d u \d t+ \int_U\E[\chi_t(\alpha^u_t)]\d u \d t,
\end{align}
{where for $\kappa$ $\in$ $\S^d_+$:
\begin{align}\label{not_K}
\Phi^u(\kappa) &= \; (A^u)\trans \kappa  + \kappa A + (C^u)\trans \kappa  C^u + Q^u \\
U^u(\kappa) & = \; (B^u)\trans \kappa + (D^u)\trans \kappa C^u \\
O^u(\kappa) & = \; R^u + (D^u)\trans \kappa  D^u,
\end{align}
for $\bar\kappa\in L^2(U\times U;\R^{d\times d})$:
\begin{align}\label{not_barK}
\Psi(K_t,\bar\kappa)(u,v) &= \;   K^u_tG_{A}(u,v)  + G_{A}(v,u)\trans K^v_t + (C^u)\trans K^u_t  G_{C}(u,v) \\
    &\qquad\qquad + \;  G_{C}(v,u)\trans K^v_tC^v+ \int_UG_{C}(w,u)\trans K^w_t  G_{C}(w,v)\d w\\
    &\qquad\qquad + \;  (A^u)\trans \bar \kappa(u,v) + \bar \kappa(v,u)\trans A^v + \int_U\bar \kappa(w,u)\trans  G_{A}(w,v)\d w \\ 
    &\qquad\qquad + \; \int_U G_{A}(w,u)\trans \bar\kappa(w,v)\d w + G^S_Q(u,v), \\
V(K_t,\bar\kappa)(u,v) &= \; (B^u)\trans\bar\kappa(u,v) + (D^u)\trans K_t^u G_C(u,v).
\end{align}
To ease the notation, we set: 
\begin{align}\label{eq:notations_Phi}
&\Phi_t^u \; = \;  \Phi^u(K^u_t), \quad U_t^u \; = \; U^u(K_t^u), \quad O_t^u \; = \; O^u(K_t^u), \\
&\Psi_t(u,v) \; = \; \Psi(u,v)(K_t,\bar K_t), \quad V_t(u,v) \; = \; V(u,v)(K_t,\bar K_t),\quad 
\Gamma_t^u \; = \; \Gamma^u(K_t,Y_t),\\
& \chi_t(\alpha^u_t) := 2\langle U^u_tX^u_t + \int_UV_t(u,v)\bar X^v_t\d v,\alpha^u_t \rangle + \langle\alpha^u_t,O^u_t\alpha^u_t\rangle.
\end{align}
}
We  now complete the square with respect to $\alpha$:
\begin{align}
    \int_U\E\left[ \chi_t(\alpha^u_t)\right]\d u =& \int_U\E\bigg[ \bigg\langle O^u_t\bigg(\alpha^u_t+(O^u_t)^{-1}U^u_tX^u_t + (O^u_t)^{-1}\int_UV_t(u,v)\bar X^v_t\d v \bigg),\\
    &\alpha^u_t+(O^u_t)^{-1}U^u_tX^u_t + (O^u_t)^{-1}\int_UV_t(u,v)\bar X^v_t\d v \bigg\rangle\\
    & - \langle U^u_tX^u_t, (O^u_t)^{-1}U^u_tX^u_t\rangle - \int_U\langle U^u_t\bar X^u_t, (O^u_t)^{-1}V_t(u,v)\bar X^v_t\rangle\d v\\
    & -\int_U\langle V_t(v,u)\bar X^u_t,(O^u_t)^{-1}U^v_t\bar X^v_t\rangle\d v - \int_U\int_U\langle V_t(w,v)\bar X^v_t, (O^w_t)^{-1}V_t(w,u)\bar X^u_t\rangle\d w\d v \bigg]\d u.
\end{align}
Substituting this in the previous equation,  we obtain:
\begin{equation}\label{eq:derivation_riccati}
\begin{aligned}
&\d \left( \mathcal{V}_t + \int_0^t\int_U\E\left[ \langle X^u_s,Q^uX^u_s \rangle + \langle \bar X^u_s,  \int_UG_Q(u,v)\bar X^v_s\d v\rangle + \langle \alpha^u_s,R^u\alpha^u_s \rangle \right]\d u\d s \right)\\
    &=\int_U\E\bigg[\langle(\dot K^u_t + \Phi_t^u - (U^u_t)^T(O^u_t)^{-1}U^u_t)X^u_t,X^u_t\rangle\bigg]\d u \d t +\\
    &\qquad\int_U\int_U\E\big[ \langle(\dot{\bar K}_t(u,v) + \Psi_t(u,v) - (U^u_t)^T(O^u_t)^{-1}V_t(u,v) -V_t(v,u)^T(O^v_t)^{-1}U^v_t \\
    &\qquad\qquad\qquad\qquad - \int_U V_t(w,u)^T(O^w_t)^{-1}V_t(w,v))\bar X^v_t,\bar X^u_t\rangle \big]\d v\d u \d t\\
    &\quad+ \int_U\E\big[\langle O^u_t(\alpha^u_t +  (O^u_t)^{-1}U^u_tX^u_t + (O^u_t)^{-1} \int_U V_t(u,v)\bar X^v_t \d v),\\
    &\qquad\qquad\alpha^u_t + (O^u_t)^{-1}U^u_tX^u_t + \int_U (O^u_t)^{-1}V_t(u,v)\bar X^v_t \d v \rangle\big] \d u \d t.
\end{aligned}
\end{equation}

{Hence, by considering the following} new system of four infinite dimensional equations related to our control problem: two equations are of Riccati type and two are linear, {we will be able state  the {\it so called} fundamental relation of the linear quadratic optimal control problem (see Proposition \ref{prop:fundamental_rel}), i.e. the decomposition of the cost functional $J$ into the sum of a quadratic term in the control and a constant term,} which then yields the optimal control.


\paragraph{Standard Riccati.} We start by introducing the system of Riccati equations: 
\begin{align}\label{eq:riccati_Ku}
    &\dot K^u_t + \Phi^u(K_t^u) - U^u(K_t^u)\trans O^u(K^u_t)^{-1}U^u(K^u_t) =0, \\
    &K^u_T = H^u, \qquad\textit{for a.e. }u \in U;
\end{align}
Notice that for each $u$ $\in$ $U$, the Riccati equation for $K^u$ in \eqref{eq:riccati_Ku} is standard. We say that a solution to \eqref{eq:riccati_Ku} is a function $K\in C^1([0,T];L^\infty(U;\S^d_+))$ such that \eqref{eq:riccati_Ku} holds for every $t\in[0,T]$.

{
\begin{Remark}\label{remark:invO}
    Observe that, for all $\kappa\in\S^d_+$, the inverse of $O^u(\kappa)$ is well defined for almost every $u\in U$ with $O^u(\kappa)^{-1}\le M$, for some positive constant $M$ independent of $u$, thanks to \eqref{ass:posR}.
\end{Remark}
}

\paragraph{Abstract Riccati.} Given a solution $K$ to \eqref{eq:riccati_Ku} (see next section), we  introduce the following abstract Riccati equation 
on the Hilbert space $L^2(U \times U;\mathbb R^{d\times d})$: 
\begin{equation}\label{eq:abstract_riccati_barK}
    \dot{\bar K}_t+ F(t,\bar K_t)=0, \quad \bar K_T= {G^S_H} ,
\end{equation}
where $F \colon [0,T] \times L^2(U \times U;\mathbb R^{d\times d}) \to L^2(U \times U;\mathbb R^{d\times d})$ is defined by 
\begin{align}
F(t,\bar\kappa)(u,v)&:= \Psi(u,v)(K_t,\bar\kappa) - U^u(K^u_t)\trans O^u(K^u_t)^{-1}V(u,v)(K_t,\bar\kappa) \\
&\qquad - \;  V(v,u)(K_t,\bar\kappa)\trans O^v(K^v_t)^{-1}U^v(K^v_t) \\
&\qquad - \int_U V(w,u)(K_t,\bar\kappa)\trans O^w(K^w_t)^{-1}V(w,v)(K_t,\bar\kappa){\d w}, \label{defF} 
\end{align}
for $\bar\kappa$ $\in$ $L^2(U \times U;\mathbb R^{d\times d})$, $t \in [0,T]$, for a.e. $u,v \in U$ {and we have used notation \eqref{eq:GS}}.
A solution to \eqref{eq:abstract_riccati_barK} is a function  $\bar K$ $\in$ $C^1([0,T];L^2(U \times U;\mathbb R^{d\times d}))$ such that  \eqref{eq:abstract_riccati_barK} holds for every $t \in [0,T]$. 
Note that  \eqref{eq:abstract_riccati_barK} is not an operator-valued Riccati equation as the ones typically considered in the classical literature on Hilbert spaces \cite{tessitore1992, guattes2005,hutang2022}.

\begin{Remark}\label{rem:properties_F}
Notice that  the function $F$ in \eqref{defF} 
satisfies the following properties:
\begin{enumerate}[(i)]
    \item   $|\mathbf F(t,\bar\kappa)|_{\Lc} \leq M_T (1+{|\bar {\boldsymbol\kappa}|_{\Lc}+|\bar {\boldsymbol\kappa}|_{\Lc}^2}), \quad   t \in [0,T], \bar\kappa \in L^2(U \times U;\mathbb R^{d \times d})$,
   {where $\mathbf F(t,\bar\kappa), \bar {\boldsymbol\kappa} $ denote the integral operator associated to the kernel $F(t,\bar\kappa), \bar \kappa$, respectively (recall notation \eqref{eq:Hilbert-schmidt_integral}).}
    \item $F$ is continuous on $[0,T] \times L^2(U \times U;\mathbb R^{d\times d});$
  \item For all $t \in [0,T]$, the map $F(t,\cdot ) \colon L^2(U \times U;\mathbb R^{d\times d}) \to L^2(U \times U;\mathbb R^{d\times d})$ is locally Lipschitz continuous, uniformly in $t$; {Moreover, the stronger following condition holds: there exists $M_T>0$} such that $\forall t \in [0,T], \bar\kappa,\bar\kappa' \in L^2(U \times U;\mathbb R^{d \times d}),$ it holds
  $$|F(t,\bar\kappa)-F(t,\bar\kappa')|_{L^2_{U\times U}} \leq {M(1+|\bar {\boldsymbol\kappa}|_{\Lc}+|\bar {\boldsymbol\kappa}'|_{\Lc})} |\bar\kappa-\bar\kappa'|_{L^2_{U\times U}}.$$
  \item {$F(t,\bar\kappa)(v,u)\trans = F(t,\bar\kappa\trans)(u,v)$},  for all $t\in [0,T], \bar\kappa \in L^2(U \times U;
    \mathbb R^{d \times d}),  (u,v)\in U\times U$. 
\end{enumerate}
All these properties follow from the {$L^2$-regularity of all the coefficients, the regularity of $K_t^u$, and the fact that for the quadratic term  in $\bar\kappa$, i.e. $\int_U\bar\kappa(u,w)\bar\kappa(w,v)\d w$, we have the following estimates. In particular  to prove (i),  by defining  the operator associated to the quadratic term in $\bar\kappa$, we have 
\begin{align*}
     (\boldsymbol{T_{\bar \kappa}} f)(u):&=\int_U\left(\int_U\bar\kappa(u,w)\bar\kappa(w,v)\d w\right)f(v)\d v\\
     &=\int_U \bar\kappa(u,w)\int_U\bar\kappa(w,v)f(v)\d v\d w = (\boldsymbol{\bar\kappa}  \circ \boldsymbol{\bar \kappa} ) (f)(u),\quad   f \in L^2(U;\mathbb R^d).
\end{align*}
Then  (i) follows from the estimate 
\begin{align}\label{eq:est_TkleqkkL2}
    \left|\boldsymbol{T_{\bar \kappa}} \right|_{\Lc}\le \|\boldsymbol{\bar \kappa} \|^2_{\Lc}.
\end{align}
Instead (iii) follows from the fact that for the quadratic term in $\bar\kappa$, we have an estimate of the type
\begin{align}
\left |\int_U \bar\kappa^1(\cdot,w) \bar\kappa^2(w,\cdot) \d w \right |_{L^2(U \times U;\mathbb R^{d \times d})} &=\int_{U}\int_{U}\left |\int_U \bar\kappa^1(u,w) \bar\kappa^2(w,v) \d w \right |^2dudv\\
&\leq \|\boldsymbol{\bar \kappa}^1 \|_{\Lc} \int_U \int_U  |\bar \kappa^2(u,v)|^2 dudv=\|\boldsymbol{\bar \kappa}^1 \|_{\Lc}  |\bar\kappa^2|_{L^2_{U \times U}},
\end{align}
for all $\bar \kappa^1, \bar \kappa^2 \in L^2(U \times U;\mathbb R^{d\times d})$, from which we derive an estimate of the form
\begin{align}
&\left |\int_U \bar\kappa(\cdot,w) \bar\kappa(w,\cdot)-\bar\kappa'(\cdot,w) \bar\kappa'(w,\cdot)  \d w \right |_{L^2_{U\times U}}\\
&\leq \left |\int_U \bar\kappa(\cdot,w) [\bar\kappa(w,\cdot)-\bar\kappa'(w,\cdot)]  \d w \right |_{L^2_{U\times U}}+\left |\int_U [\bar\kappa(\cdot,w)-\bar\kappa'(\cdot,w)] \bar\kappa'(w,\cdot)  \d w \right |_{L^2_{U\times U}}\\
&\leq \left(\|\boldsymbol{\bar \kappa} \|_{\Lc} +\|\boldsymbol{\bar \kappa}' \|_{\Lc} \right)   |\bar\kappa'-\bar\kappa|_{L^2_{U \times U}}
\end{align}
}
{
For point (iv), we show the property for the term $\Psi$, the other terms can be treated in an analogous way. Let us consider
\begin{align}
    \Psi(K_t,\bar\kappa)(v,u)\trans &= \;   G_{A}(v,u)\trans (K^v_t)\trans + (K^u_t)\trans G_{A}(u,v) + \;  G_{C}(v,u)\trans (K^v_t)\trans C^v  \\
    &\qquad\qquad +  (C^u)\trans (K^u_t)\trans  G_{C}(u,v) + \int_UG_{C}(w,u)\trans (K^w_t)\trans  G_{C}(w,v)\d w\\
    &\qquad\qquad + \;   \bar \kappa(u,v) A^v +  (A^u)\trans\bar \kappa(v,u)\trans + \int_U  G_{A}(w,u)\trans \bar \kappa(w,v)\d w \\ 
    &\qquad\qquad + \; \int_U  \bar\kappa(w,u)\trans G_{A}(w,v)\d w + G^S_Q(v,u)\trans\\
    &= G_{A}(v,u)\trans K^v_t + K^u_t G_{A}(u,v) + \;  G_{C}(v,u)\trans K^v_t C^v  \\
    &\qquad\qquad +  (C^u)\trans K^u_t  G_{C}(u,v) + \int_UG_{C}(w,u)\trans K^w_t  G_{C}(w,v)\d w\\
    &\qquad\qquad + \;  (A^u)\trans \bar \kappa(v,u)\trans + \bar \kappa(u,v) A^v + \int_U\bar \kappa(w,u)\trans  G_{A}(w,v)\d w \\ 
    &\qquad\qquad + \; \int_U G_{A}(w,u)\trans \bar\kappa(w,v)\d w + G^S_Q(u,v)\\
    & = \Psi(K_t,\kappa\trans)(u,v)
\end{align}
where we have used the fact that $K^u\in\S^d$ for all $u\in U$ in the second equality.
}
\end{Remark}


\begin{Remark}\label{rem:symbarK}
\begin{enumerate}[(i)]
    \item From \cite{zeidler} with Remark \ref{rem:properties_F}, we know that  a solution to \eqref{eq:abstract_riccati_barK} (when it exists, see next section)  is unique. 
\item We notice  that, if $\bar K_t$ is a solution to \eqref{eq:abstract_riccati_barK}, then it satisfies
\begin{align} 
\bar K_t(v,u)\trans &= \; \bar K_t(u,v), \quad \forall t\in [0,T], (u,v)\in U\times U.
\end{align} 
{Indeed exchanging $u$ and $v$ in \eqref{eq:abstract_riccati_barK},  taking the transpose, we have that 
\begin{align}
0= ( \dot{\bar K}_t(v,u)+ F(t,\bar K_t)(v,u))\trans=  \dot{\bar K}_t\trans(v,u)+ F(t,\bar K_t)\trans(v,u) = \dot{\bar K}_t\trans(v,u)+ F(t,\bar K_t\trans)(u,v),
\end{align}
where we have used  Remark \ref{rem:properties_F} (iv).} Then we have that $\bar K_t(v,u)\trans$ satisfies the same equation \eqref{eq:abstract_riccati_barK}. Then the claim follows from uniqueness of solutions to \eqref{eq:abstract_riccati_barK}.
\end{enumerate}
\end{Remark}

\paragraph{Linear equations.} Given $K,\bar K$ solution to \eqref{eq:riccati_Ku}, \eqref{eq:abstract_riccati_barK},  let us introduce the  terminal value problem on $L^2(U;\R^d)$:
\begin{align}\label{eq:lineqY_L2}
    \dot Y_t + \tilde F(t,Y_t) = 0,\quad Y_T = 0,
\end{align}
where $\tilde F : [0,T]\times L^2(U;\R^d)\to L^2(U;\R^d)$ is defined by 
\begin{align}
    \tilde F^u(t,y) &= \;  (A^u)\trans y^u + \int_U G_{A}(v,u)\trans y^v\d v +  2K^u_t \beta^u + 2(C^u)\trans K^u_t\gamma^u \\
    & \quad  + \;  2\int_U G_{C}(v,u)\trans K^v_t\gamma^v\d v + 2 \int_U \bar K_t(u,v)\beta^v\d v\\
    & \quad - \;  U^u(K^u_t)\trans O^u(K^u_t)^{-1}\Gamma^u(K_t,y) - \int_U 
    V(v,u)(K_t,\bar K_t)\trans O^v(K^v_t)^{-1}\Gamma^u(K_t,y)\d v,
\end{align}
for all $t\in[0,T],y\in L^2(U,\R^d)$, for a.e. $u\in U$, with 
\begin{align}
\Gamma^u(K_t,y) &= \; (D^u)\trans K_t^u \gamma^u + (B^u)\trans y^u.     
\end{align}
Equation \eqref{eq:lineqY_L2} is a standard linear ordinary differential equation (ODE) on the Hilbert space $L^2(U;\R^d)$, and admits a unique solution $Y$ $\in$ $C^1([0,T];L^2(U;\R^d))$, see details in the next section.  Finally, we introduce the linear equation 
\begin{align}\label{eq:lineqLambda}
    &\dot \Lambda^u_t + \langle K^u_t\gamma^u,\gamma^u\rangle + \langle Y^u_t,\beta^u\rangle 
    - \langle \Gamma^u_t(K_t,Y_t), O^u(K_t^u)^{-1}\Gamma^u_t(K_t,Y_t)\rangle \; = \;  0,\\
    &\Lambda^u_T \; = \; 0, \qquad\textit{for a.e. }u \in U,
\end{align}
which admits a unique solution $\Lambda$ $\in$ $C^1([0,T];L^\infty(U;\R))$.


\begin{Remark}\label{solnulle}
    If $\beta=\gamma=0$, the solutions to the linear equations \eqref{eq:lineqY_L2} and \eqref{eq:lineqLambda} are 
    $Y$ $\equiv$ $\Lambda$ $\equiv$ $0$.  
\end{Remark}

\begin{Remark}\label{rem:system_of_differential_equations}
\begin{enumerate}[(i)]
\item The system of equations \eqref{eq:riccati_Ku}-\eqref{eq:abstract_riccati_barK}-\eqref{eq:lineqY_L2}-\eqref{eq:lineqLambda}   has a triangular form, i.e. the equation for $K$ is independent of $\bar K, Y, \Lambda$; the equation for $\bar K$ depends on $K$ but it is independent of $Y, \Lambda$; the equation for $Y$ depends on $K, \bar K$ but it is independent of $\Lambda$; finally, the equation for $\Lambda$ depends on $K, \bar K, Y$.
\item Equations \eqref{eq:riccati_Ku}-\eqref{eq:abstract_riccati_barK}-\eqref{eq:lineqY_L2}-\eqref{eq:lineqLambda} are infinite dimensional; however, \eqref{eq:riccati_Ku}-\eqref{eq:lineqLambda} are decoupled in $u$.
 \item  The Riccati equations \eqref{eq:riccati_Ku}, \eqref{eq:abstract_riccati_barK} are independent of $\beta,\gamma$.
\end{enumerate}
\end{Remark}

\vspace{2mm}

The next Section \ref{sec:existence} is devoted to the existence (and uniqueness) of a solution to the system of equations \eqref{eq:riccati_Ku}-\eqref{eq:abstract_riccati_barK}-\eqref{eq:lineqY_L2}-\eqref{eq:lineqLambda}.

\vspace{2mm}

We end this section by stating the {\it so-called} fundamental relation of the optimal control problem. Such result states  that  we  can decompose  the functional $J_{}(t,\xi,\alpha)$ into a  constant  part (not depending on the control) and a quadratic non-negative part,  and  therefore will be crucial to solve the optimal control problem  (see Section  \ref{sec:verif_thm}).

\begin{Proposition}[Fundamental relation]\label{prop:fundamental_rel}
Let $K,\bar K,Y,\Lambda$ be solution to the system \eqref{eq:riccati_Ku}, \eqref{eq:abstract_riccati_barK}, \eqref{eq:lineqY_L2}, \eqref{eq:lineqLambda}. Then, given $t \in [0,T]$, $\xi$ $\in$ $\Ic$, for all $\alpha$ $\in$ $\Ac$ we have 
\begin{small}
\begin{align}
J_{}(t,\xi,\alpha) & = \;  \mathcal{V}(t,\xi)   + \int_t^T  \int_U\E\Big[\Big\langle O^u_s\Big(\alpha^u_s + (O^u_s)^{-1}\big(U^u_sX^u_s + \int_U V_s(u,v)\bar X^v_s \d v + \Gamma_s^u\big)\Big),   \\
& \qquad \qquad  \qquad \qquad \quad \quad \alpha^u_s + (O^u_s)^{-1}\big(U^u_sX^u_s + \int_U V_s(u,v)\bar X^v_s \d v + \Gamma_s^u \big)\Big\rangle \Big]\d s, \label{eq:fundamental_relation}
\end{align}
\end{small}
where $\Vc : [0,T]\times\Ic \to \R$ is defined by
\begin{align} \label{def_v}
    \mathcal{V}(t,\xi) & := \int_U\E[\langle \xi^u, K^u_t \xi^u \rangle]\d u 
    + \int_U\int_U\langle\bar \xi^u,\bar K_t(u,v)\bar \xi^v\rangle\d v \d u \\
    &\qquad\qquad\qquad\qquad + \; 2 \int_U \E[\langle  Y^u_t,\xi^u\rangle] \d u + \int_U \Lambda^u_t\d u.
\end{align} 
\end{Proposition}
\begin{proof}
Recalling the form of $\mathcal{V}$ in \eqref{def_v} (with $Y^u=\Lambda^u=0$), integrating \eqref{eq:derivation_riccati} over $[t,T],$ and using \eqref{eq:riccati_Ku}, \eqref{eq:abstract_riccati_barK},  we have
\begin{align*}
& \int_t^T  \int_U\E\big[\langle O^u_s(\alpha^u_s + (O^u_s)^{-1}U^u_sX^u_s + (O^u_s)^{-1}\int_U V_s(u,v)\bar X^v_s \d v ),\\
&\quad \quad \quad \quad \alpha^u_s + (O^u_s)^{-1}U^u_sX^u_s + (O^u_s)^{-1 }\int_UV_s(u,v)\bar X^v_s \d v\rangle\big]\d u \d s\\
 &= \mathcal{V}_T -\mathcal{V}_t +\int_t^T\int_U\E\bigg[ \langle X^u_s,Q^uX^u_s \rangle + \langle \bar X^u_s,  \int_UG_Q(u,v)\bar X^v_s\d v\rangle + \langle \alpha^u_s,R^u\alpha^u_s \rangle \bigg]\d u\d s \\
 &=\int_U\E\left[ \langle X^u_T, H^uX^u_T\rangle \right]\d u + \int_U\int_U\langle\bar X^u_T, G^S_H(u,v) \bar X^v_T\rangle\d v \d u - \mathcal{V}(t,\xi)\\
& \quad +\int_t^T\int_U\E\bigg[ \langle X^u_s,Q^uX^u_s \rangle + \langle \bar X^u_s,  \int_UG^S_Q(u,v)\bar X^v_s\d v\rangle + \langle \alpha^u_s,R^u\alpha^u_s \rangle \bigg]\d u\d s \\
&=\int_U\E\left[ \langle X^u_T, H^uX^u_T\rangle \right]\d u + \int_U\int_U\langle\bar X^u_T, G_H(u,v) \bar X^v_T\rangle\d v \d u - \mathcal{V}(t,\xi)\\
& \quad +\int_t^T\int_U\E\bigg[ \langle X^u_s,Q^uX^u_s \rangle + \langle \bar X^u_s,  \int_UG_Q(u,v)\bar X^v_s\d v\rangle + \langle \alpha^u_s,R^u\alpha^u_s \rangle \bigg]\d u\d s \\
&= J(t,\xi,\alpha)- \mathcal{V}(t,\xi),
\end{align*}
where we have used the equality $\Vc_t = \Vc(t,\xi)$ and \eqref{eq:symmetric_structure_graphon_cost_term_rem}.
The statement of the proposition follows.
\end{proof}

\section{Solution of the infinite dimensional differential system} \label{sec:existence}

In this section, we solve the system of infinite-dimensional equations   \eqref{eq:riccati_Ku}-\eqref{eq:abstract_riccati_barK}-\eqref{eq:lineqY_L2}-\eqref{eq:lineqLambda}. Since such system has a triangular form (see Remark \ref{rem:system_of_differential_equations}) we will solve the system by proceeding with one equation at a time, starting with \eqref{eq:riccati_Ku} and ending with \eqref{eq:lineqLambda}.
\subsection{Standard Riccati equation}

To ensure existence and uniqueness of a solution of \eqref{eq:riccati_Ku}, we can rely on the classical theory. 
Indeed, for a.e. $u\in U$, equation \eqref{eq:riccati_Ku} is a strong Riccati equation.
We recall that our coefficients are $L^\infty$ functions of $u$, hence they are defined up to a null set. However, we can still look at the Riccati equation for (a.e.) $u\in I$ fixed. 
Indeed, let us fix a choice of representatives $\tilde A,\tilde B,\dots$ of the coefficients $A,B,\dots$, defined out of a null set $\tilde N$, which is the same for all the representatives. 
Then, for all $u\in \tilde N^c$ we can consider the associated strong Riccati equation and find a unique solution $\tilde K^u$.
If we consider a different choice of representatives $\hat A,\hat B,\dots$ defined out of a null set $\hat N$, then $\tilde A^u=\hat A^u,\tilde B^u=\hat B^u,\dots$ for all $u\in\tilde N^c\cap\hat N^u$, and $\tilde N^c\cap\hat N^c$ is a set of measure $1$.
The uniqueness of the solution ensures that $\tilde K^u=\hat K^u$ for all $u\in \tilde N^c\cap\hat N^c$, thus, $\tilde K$ and $\hat K$ belong to the same equivalence class in $L^\infty$.

Thus, using for example \cite[Theorem 7.2]{yong-zhou}, we have that for a.e. $u\in U$, there exists a unique solution $K^u\in C^1([0,T];\S^d_+)$, which is associated 
to the standard linear quadratic stochastic control problem with $G_{  A}=G_{C}= G_Q=G_H\equiv 0$, i.e.
\begin{equation}
    V^u(t,x):=\inf_{\alpha\in\Ac}\E\left[ \int_t^T f^u(\tilde X_s,\alpha_s) \d s + g^u(\tilde X_T)\right],
\end{equation}
with dynamics
\begin{equation}
        \d\tilde X_s = b^u(\tilde X_s,\alpha_s)\d s + \sigma^u(\tilde X_s,\alpha_s)\d {W_s^u}, \quad \tilde X_0 = x,
\end{equation}
where
$ b^u(x,a) := A^ux+B^ua,$ 
$\sigma^u(x,a):= C^ux+D^ua$, $f^u(x,a):=\langle x,Q^ux\rangle + \langle a,R^ua\rangle,$  $g^u(x):= \langle x,H^ux\rangle$.
In this case, $V^u(t,x) = \langle x,K^u_t x\rangle$, 
with $|K_t^u|\leq C_{T,r}$ where $r>0$ is such that $|A^u|, |B^u|, |C^u|,$ $|D^u|, |Q^u|,$ $ |R^u|, |H^u|\leq r.$ 
Therefore, as our coefficients are bounded with respect to $u \in U$, there exists $M_T>0$ such that
\begin{equation}\label{eq:Ku_bounded}
    |K^u_t| \le M_T \quad \forall t \in [0,T].
\end{equation}
{Following the proof of \cite[Theorem 7.2]{yong-zhou}, we recall that, for each $u\in U$, $K^u$ is obtained as the limit of a sequence of functions $(K^i)^u$ which are solutions of linear ODEs.
Explicitly, $K^0 := H^u$, and for all $i\ge 1$ $K^i$ is defined as the solution of the linear equation
\begin{align}
&(\dot K^i_t)^u + (K^i_t)^u(\hat A^{i-1}_t)^u + ((\hat A^{i-1}_t)\trans)^u (K^i_t)^u + ((\hat C^{i-1}_t)\trans)^u (K^i_t)^u(\hat C^{i-1}_t)^u +(\hat Q^{i-1}_t)^u = 0,\\
&(K^i_T)^u = H^u,
\end{align}
where
\begin{align}
    &(\hat A^{i-1}_t)^u = A^u - B^u(\hat\Psi^{i-1}_t)^u, \qquad (\hat C^{i-1}_t)^u = C^u - D^u(\hat\Psi^{i-1}_t)^u,\\
    & (\hat Q^{i-1}_t)^u = ((\hat\Psi^{i-1}_t)\trans)^uR^u(\hat\Psi^{i-1}_t)^u + Q^u\\
    & (\hat\Psi^{i-1}_t)^u = (R^u +(D\trans)^u (K^{i-1}_t)^u D^u)^{-1}((B\trans)^u (K^{i-1}_t)^u + (D\trans)^u (K^{i-1}_t)^u C^u)
\end{align}
We have that, for each $i\in\N$, $(u,t)\mapsto (K^i_t)^u$ is measurable. This can be done by exploiting the standard Picard iteration to obtain solutions of linear ODEs and by arguing that $K^i$ is the limit of measurable functions.
Indeed, we have that:
\[
(K^i_t)^u = \lim_{n\to\infty} (K_t^{i,n})^{u}, \text{ for almost every } u\in U,\text{ for all }t\in[0,T],
\]
where $$(K_t^{i,n})^u := \int_t^T\big[(K^{i,n-1}_s)^u(\hat A^{i-1}_s)^u + ((\hat A^{i-1}_s)\trans)^u (K^{i,n-1}_s)^u + ((\hat C^{i-1}_s)\trans)^u (K^{i,n-1}_s)^u(\hat C^{i-1}_s)^u +(\hat Q^{i-1}_s)^u]\d s$$ and $(K_t^{i,0})^u := K_T^u$ for all $t\in[0,T]$. 
Therefore, since $\sup_{t\in[0,T]}|K^{i,u}_t-K^u_t|\to0$ as $i$ tends to infinity, we can conclude that $(u,t)\mapsto K^u_t$ is measurable. By \eqref{eq:Ku_bounded}, we conclude that $K\in L^\infty(U,\S^d_+)$.}
{To show that $K\in C([0,T];L^\infty(U;\S^d_+))$, i.e. $K$ is continuous in $t$, we need to show that $\esssup_{u\in U}|K_{t_1}^u-K_{t_2}^u| \to 0$ as $t_1\to t_2$. Therefore, let us consider
\begin{align}
    |K_{t_1}^u-K_{t_2}^u| &= \left|\int_{t_1}^{t_2} \left( \Phi^u(K^u_s) - U^u(K^u_s)\trans O^u(K^u_s)^{-1}U^u(K^u_s)\right)\d s\right| \le M|t_1-t_2|,
\end{align}
where $M$ is a positive constant. The last inequality can be obtained thanks to \eqref{eq:Ku_bounded}, to the boundedness of all the coefficients involved, and to Remark \ref{remark:invO}. 
Finally, we show that $\dot K\in C([0,T];L^\infty(U;\S^d_+))$ by observing that $\dot K^u_t$ is given by \eqref{eq:riccati_Ku}
and therefore is the sum and product of functions in $C([0,T];L^\infty(U;\S^d_+))$
}
Therefore, we can conclude that $K\in C^1([0,T];L^\infty(U;\S^d_+))$, so that it is a solution to \eqref{eq:riccati_Ku}.

\subsection{Solution of the infinite dimensional Riccati equation}

In order to solve the infinite dimensional Riccati equation \eqref{eq:abstract_riccati_barK}, we start by proving an a priori estimate on its solutions.
As observed in  Remark \ref{rem:system_of_differential_equations}, equation \eqref{eq:abstract_riccati_barK} is independent of $\beta^u,\gamma^u$; therefore, in this subsection, without loss of generality, we assume 
\begin{align}
\beta^u=\gamma^u=0,\quad \text{for a.e. }u\in U.
\end{align}
This will be useful in deriving the a priori estimate in Proposition \ref{cor:a_priori_estimate_barK}.
\begin{Proposition}[A priori estimate]\label{cor:a_priori_estimate_barK}
Let $T_0 \in [0,T]$ and $\bar K_t$ be a solution  of the abstract  Riccati equation \eqref{eq:abstract_riccati_barK} on $(T_0,T]$. Then there exists $M_T>0$ (independent of $\bar K_t$) such that {\begin{equation}\label{eq:norm_operator_bounded}
    \|\boldsymbol{\bar K}_t\|_{\Lc}\le M_T \quad \forall t \in (T_0,T].
\end{equation}}
\end{Proposition}
\begin{proof}
First, we observe that $\boldsymbol{\bar K}_t$ is a self-adjoint operator, thanks to Remark \ref{rem:symbarK}. Indeed:
\begin{align}
    \langle z, \boldsymbol{\bar K}_t(y)\rangle_{L^2} &= \int_U\int_U(z^u)^T\bar K_t(u,v)y^v\d v\d u
    = \int_U\int_U(z^v)^T\bar K_t(v,u)y^u\d v\d u\\
    &= \int_U\int_U(z^v)^T\bar K_t(u,v)^Ty^u\d v\d u= \langle\boldsymbol{\bar K}_t(z),y\rangle_{L^2}.
\end{align}
 We recall that the norm of self-adjoint operators can be expressed in the following way:
\begin{align}\label{eq:norm_operator_bounded2}
\|\boldsymbol{\bar K}_t \|_{\Lc)}^2 = \sup_{\|f\|_{L^2}\ne 0}\frac{\langle\boldsymbol{\bar K}_t(z),z\rangle_{L^2}}{\|z\|_{L^2}^2}
\end{align}
Now, we estimate the right-hand-side.
\begin{itemize}
\item {Considering \eqref{eq:fundamental_relation}, and noticing that both the square term there and $\int_U\E\left[ \langle \xi^u, K^u_t \xi^u\rangle \right]\d u$ are non-negative, it holds
$J(t,\xi,\alpha) \geq \int_U\int_U\langle\bar \xi^u,\bar K_t(u,v)\bar \xi^v\rangle\d v \d u,$ $ \forall \alpha 
$.
Observe that, when $\beta=\gamma=0$, the solutions to the linear equations \eqref{eq:lineqY_L2},\eqref{eq:lineqLambda} are $Y=\Lambda=0$.
Taking $\alpha=0$ and using \eqref{estimJ} (with $\beta=\gamma=0$), yield
\[
M_T\int_U\E\left[|\xi^u|^2\right]\d u \geq J(t,\xi,0) \geq \int_U\int_U\langle\bar \xi^u,\bar K_t(u,v)\bar \xi^v\rangle\d v \d u.
\]}
    \item {Denoting $\alpha^u_s := -(O^u_s)^{-1}U^u_sX^u_s - (O^u_s)^{-1}\int_UV_s(u,v)\bar X^v_s\d v$, where $X^u_s$ is the unique solution to the closed loop equation \eqref{eq:closed_loop}, 
 from \eqref{eq:fundamental_relation} we have
\[
\int_U\E\left[ \langle \xi^u, K^u_t \xi^u\rangle \right]\d u +\int_U\int_U\langle\bar \xi^u,\bar K_t(u,v)\bar \xi^v\rangle\d v \d u = J(t,\xi,\alpha) \ge 0,
\]
from which we obtain, taking into account also \eqref{eq:Ku_bounded},
\[
\int_U\int_U\langle\bar \xi^u,\bar K_t(u,v)\bar \xi^v\rangle\d v \d u\ge - \int_U\E\left[ \langle \xi^u, K^u_t \xi^u\rangle \right]\d u\ge -M_T\int_U\E\left[|\xi^u|^2\right]\d u.
\]
}
\end{itemize}
{Therefore it holds 
\begin{align*}
\left|\int_U\int_U\langle\bar \xi^u,\bar K_t(u,v)\bar \xi^v\rangle\d v \d u\right|& \le M_T\int_U\E\left[|\xi^u|^2\right]\d u.
\end{align*}
Hence, taking deterministic  random variables $\xi^u=z^u$, we are able to bound \eqref{eq:norm_operator_bounded2}, so that \eqref{eq:norm_operator_bounded} follows. Hence, we have proved the  statement of the proposition.}
\end{proof}


We now solve  the abstract Riccati equation \eqref{eq:abstract_riccati_barK}.

\begin{Proposition}
The abstract Riccati equation \eqref{eq:abstract_riccati_barK} has a unique solution $\bar K_t$ on $[0,T]$. 
\end{Proposition}

\begin{proof}{Inspired by the proof of \cite[Proposition 5.8]{guattes2005}, we proceed as follows.
Let $\delta, r >0$. Consider the Banach space $C\left([T-\delta, T] ; L^2(U \times U; \mathbb R^{d \times d})\right)$ with norm $|\bar k|_C^2=\sup _{t \in[T-\delta, T]} |\bar k_t|_{L^2_{U\times U}}^2$ and consider
$$
B(r):=\left\{\bar k \in C\left([T-\delta, T] ; L^2(U \times U; \mathbb R^{d \times d})\right): \sup _{t \in[T-\delta, T]}|\mathbf {\bar k}_t|_{\Lc} \leq r \right\},
$$
where $\mathbf {\bar k}_t$ is the integral operator associated to $\bar k_t$  (recall notation \eqref{eq:Hilbert-schmidt_integral}) . We notice that $B(r)$ is a closed convex subset of $C\left([T-\delta, T] ; L^2(U \times U; \mathbb R^{d \times d})\right)$. As convexity is immediate, we prove that it is closed. Indeed, let $\bar k^n \in B(r), \bar k \in C\left([T-\delta, T] ; L^2(U \times U; \mathbb R^{d \times d})\right)$ such that $|\bar k^n-\bar k|_C \to 0$. Then we have $\bar k \in B(r)$ since (with usual notations for $\mathbf {\bar k}^n$)
$$\sup _{t \in[T-\delta, T]}|\mathbf {\bar k}_t|_{\Lc}\leq \sup _{t \in[T-\delta, T]} |\mathbf {\bar k}_t-\mathbf {\bar k}^n_t|_{\Lc}+\sup _{t \in[T-\delta, T]}|\mathbf {\bar k}^n_t|_{\Lc}\leq |\bar k-\bar k^n|_C+r\xrightarrow{n\to \infty}r.$$ The claim follows.}

{On $B(r)$ we construct a map (recall \eqref{eq:abstract_riccati_barK})
$$
\mathcal A: B(r) \rightarrow B(r), \quad (\mathcal A\bar K)_t= \bar K_T+\int_t^T F\left(s, \bar{K}_s\right)ds, \quad t \in [T-\delta,T].
$$ 
We show first that $\mathcal A$ is well defined for opportune choices of the parameters. Indeed,  denoting $\mathbf{(\mathcal A\bar K)}$ the integral operator associated to $\mathcal {A}\bar K$ (recall notation \eqref{eq:Hilbert-schmidt_integral})  and using Remark \ref{rem:properties_F} (i),  for $M=M_T>0$ (independent of $\delta, r$), we have
\begin{align}
|\mathbf{(\mathcal A\bar K)}_t|_{\Lc} &\leq M \int_{T-\delta}^T(1+ |\mathbf {\bar K}_s|_{\Lc} +|\mathbf {\bar K}_s|_{\Lc} ^2  )ds+ |\mathbf {\bar K}_T|_{\Lc} \\
&\leq M\delta (1+r +r^2 ) + |\mathbf {\bar K}_T|_{\Lc} , \quad t \in [T-\delta,T].
\end{align}
Therefore, taking first $r >  |\mathbf {\bar K}_T|_{\Lc}$ and then $\delta $ such that $M\delta (1+r +r^2 ) + |\mathbf {\bar K}_T|_{\Lc} \leq r$, we have the claim.}

{Finally we show that $\mathcal A$ is a contraction on $B(r)$. Indeed, using  Remark \ref{rem:properties_F} (iii), for $M=M_T>0$ (independent of $\delta,r$) we have
\begin{align}
    \sup_{t \in [T-\delta,T]} |(\mathcal A\bar K')_t- (\mathcal A\bar K)_t|_{L^2} &\leq M \int_{T-\delta}^T (1+|\mathbf {\bar K}_s|_{\Lc}+|\mathbf {\bar K}'_s|_{\Lc}) | \bar K'_s-\bar K_s|_{L^2_{U\times U}} ds\\
    &\leq  \delta M(1+2r) \sup_{t \in [T-\delta,T]} |\bar K'_t- \bar K_t|_{L^2_{U\times U}} .
\end{align}
Then choosing $\delta $ such that $\delta M(1+2r) <1$ (together with the previous condition $M\delta(1+r  +r^2  ) + |\mathbf {\bar K}_T|_{\Lc} \leq r$) we have the claim. 
 } 
 
{ This provides existence and uniqueness of the solution $\bar K_t$ on  $[T-\delta,T]$. Observing that $\delta$ depends on $\bar K_T$ only through $|\bar{\mathbf K}_T|_{\Lc}$ for which we have the a-priori estimate \eqref{eq:norm_operator_bounded}, we can iterate the argument backward in time and get global existence an uniqueness on $[0,T]$.}
\end{proof}
\subsection{Solution of the linear equations}\label{sec:sollineq}

Applying \cite[Corollary 3.8]{zeidler}, we get existence and uniqueness of the solution of \eqref{eq:lineqY_L2}. 
Equation \eqref{eq:lineqLambda} can be solved $u$ by $u$, and  the classical theory for differential equations grants existence and uniqueness.

\section{Optimal feedback controls} 
\label{sec:verif_thm}

In this section, we prove a verification theorem.
In view of the fundamental relation, this suggests that the optimal control satisfies the following feedback form  
\begin{equation}\label{eq:opt-control}
\alpha^u_s = -(O^u_s)^{-1}\Big( U^u_s X^u_s + \int_UV_s(u,v)\bar X^v_s\d v + \Gamma^u_s\Big), 
\end{equation}
and to consider the following closed loop equation (by substituting such $\alpha$ in the state dynamics \eqref{dynamics}):
\begin{small}
\begin{equation}\label{eq:closed_loop}
\begin{aligned}
    &\d X^u_s =\big [ \left (A^u - B^u(O^u_s)^{-1}U^u_s \right ) X^u_s +  \int_U \left(G_{A}(u,v)-B^u (O^u_s)^{-1}V_s(u,v)\right)\bar X^v_s \d v -B^u(O^u_s)^{-1}\Gamma^u_s\big] \d s\\
    & \qquad  + \big [ \left ( C^u-D^u (O^u_s)^{-1}U^u_s \right)  X^u_s+ \int_U \left(G_{C}(u,v)- D^u (O^u_s)^{-1} V_s(u,v)\right)\bar X^v_s\d v -D^u(O^u_s)^{-1}\Gamma^u_s\big ] \d W^u_s,\\ 
    & X_t^u=\xi^u, \quad u \in U.  
\end{aligned}
\end{equation}
\end{small}
Since \eqref{eq:closed_loop} is linear mean-field SDE  with coefficients which are  continuous in time and bounded, by Remark \ref{rem:non-autonomous_coeff},  it admits a unique solution, denoted  $\hat X^u_s$ on $[t,T]$. 
The next result shows the optimality of the control in \eqref{eq:opt-control}.

\begin{Theorem}\label{th:verification}
Let $K_t^u$ be the unique solution of the standard Riccati equation \eqref{eq:riccati_Ku} on $[0,T]$, $\bar K_t$ be the unique solution  of the abstract Riccati equation \eqref{eq:abstract_riccati_barK} on $[0,T]$, $Y,\Lambda$ {be} the unique solutions of the linear equations \eqref{eq:lineqY_L2},\eqref{eq:lineqLambda}, respectively, on $[0,T]$. Then:
\begin{enumerate}
    \item The unique optimal control is given by the feedback control
    \begin{align}\label{eq:feedback_control}
        \hat \alpha^u_s: &=- (O^u_s)^{-1}\left(U^u_sX^u_s + \int_U V_s(u,v)\bar X^v_s dv + \Gamma^u_s\right),
    \end{align}
for every $s \in [0,T]$, where $X^u_s$ is the unique solution to the closed loop equation \eqref{eq:closed_loop};
\item  The value function $V(t,\xi):=\Inf_{\alpha\in\Ac} J_{}(t,\xi,\alpha)$ is given by $V(t,\xi)=\mathcal{V}(t,\xi)$ (recall \eqref{def_v}).
\end{enumerate}
\end{Theorem}
\begin{proof}
Observe that $\hat\alpha\in\Ac$ (this follows from the proof of Theorem \ref{exuniqsol} in Appendix \ref{app:exuniqsol}).  By the fundamental relation \eqref{eq:fundamental_relation}, we have
\begin{itemize}
    \item $J_{}(t,\xi,\alpha) \geq \mathcal{V}(t,\xi)$,  for all $\alpha$ $\in$ $\Ac$; 
    \item we have equality, i.e., 
$J_{}(t,\xi,\alpha) = \mathcal{V}(t,\xi),$
if and only if a control $\alpha$ satisfies \eqref{eq:feedback_control}.  Hence, by uniqueness of solutions to the closed loop equation \eqref{eq:closed_loop}, we have equality if and only if $\alpha= \hat \alpha.$
\end{itemize}
This implies the statement of the theorem.
\end{proof}

\begin{Remark}
    Observe that in the homogeneous case, i.e. when $G_A,G_C,G_Q,G_I,G_H$ are constant in $(u,v)$, we retrieve the results for the classical LQ mean-field control (see for example \cite{yong13}, \cite{huangetal15}, \cite{sunyong17}, \cite{baseipham2019}).
\end{Remark}

\section{Application to systemic risk with heterogeneous banks}\label{sec:systemic_risk}

In this section,  we introduce and solve a systemic risk problem with heterogeneous banks, extending the model proposed in \cite{systemicrisk} with homogeneous interaction, i.e., in a standard linear-quadratic McKean-Vlasov control problem. 
See also \cite{sun-hetsysrisk}, for the mean field game case, with homogeneity within groups and heterogeneity between groups. 

In what follows, all the quantities involved are real-valued.
Consider a model of interbank borrowing and lending of $N$ heterogenous banks, with the log-monetary reserve of each bank $i $ given by \begin{align} \d X^i_s &= \; \Big[\frac {k} {N_i} \sum_{j=1}^N G(u_i,u_j)(X_t^i-X_t^j)+\alpha^i_s \Big] \,\d s + \sigma^i\,\d W^i_s,\qquad X^i_t = \xi^i.\end{align} 
Here, $k \le 0$, $G$ is a graphon i.e. a bounded, symmetric, measurable function from $U\times U$ into $\mathbb R$, measuring the rate of borrowing/lending between bank $i$ and bank $j$,  $N_i= \sum_{j=1}^N G(u_i,u_j)$, with $u_i=i/N$,  and $\sigma^i > 0$ is the volatility coefficient of the bank reserve. 
All banks can control their rate of borrowing/lending to a central bank with the collection of  policies $\alpha = (\alpha^i)_{i=1}^N$.

Taking heuristically the limit for $N\to\infty$, we are led to study a continuum of heterogeneous banks, indexed by $u\in U$, where the log-monetary reserve of each bank $u\in U$ is governed by 
\begin{align} 
\d X^u_s &= \; \Big[
k\big(
X^u_s -\int_U \tilde G_{k}(u,v)\,\bar X^{v}_s\,\d v \big)
+\alpha^u_s \Big] \,\d s + \sigma^u\,\d W^u_s,\qquad X^u_t = \xi^u,
\end{align} 
where $\tilde G_k$ is a graphon i.e. a bounded, symmetric, measurable function from $U\times U$ into $\mathbb R$, $\sigma^u>0$, and $\alpha=(\alpha^u)_u$ is the collection of policies.

The aim of the central bank is to mitigate systemic risk by  minimizing over all $\alpha$ an aggregate cost functional of the form
\begin{small}
\begin{align*}
\int_U \E\Big[&
\int_t^T\Big( 
\eta^u\big(
X^u_s -\int_U \tilde G_\eta(u,v)\,\bar X^{v}_s\,\d v
\big)^2  + (\alpha^u_s)^2\Big) \d s + r^u\big(
X^u_T -\int_U \tilde G_r(u,v)\,\bar X^{v}_T\,\d v
\big)^2 \Big]\, \d u,
\end{align*}
\end{small}
where 
$ \eta^u > 0, r^u > 0$ are positive parameters for penalizing departure from the weighted average, and the term $(\alpha^u_s)^2$ represents the cost of borrowing/lending from the central bank.

\vspace{1mm}

We have shown in Appendix \ref{subsec:centered_formulation}, how to rewrite the above functional in the form \eqref{eq:J_theory}. Here, we have
\begin{small}
\begin{align*}
J(t,\xi,\alpha)=&\int_U \E\bigg[
\int_t^T\left(  
\eta^u (X^u_s)^2+ \bar X^{u}_s \int_U G_\eta(u,v)  \bar X^{v}_s  \d v 
+ (\alpha^u_s)^2 \right)\d s\\
&\qquad\qquad\qquad\qquad+ r^u(X^u_T)^2 +\bar X^{u}_T \int_U G_r(u,v) \bar X^{v}_T  \d v
\bigg]\,\d u,
\end{align*}
\end{small}
where $G_\eta(u,v):=\int_U \eta^w  \tilde G_\eta(w,v)\tilde G_\eta(w,u) \d w -  (\eta^u+\eta^v)  \tilde G_\eta(u,v)$,  $G_r(u,v)$ $:=$ $\int_U r^w  \tilde G_r(w,v)\tilde G_r(w,u) \d w$ $-$ $( r^u +r^v) \tilde G_r(u,v)$. Hence,  we can apply  Theorem \ref{th:verification} to solve the optimal control problem. In particular, the value function is in the quadratic  form
\begin{align} 
V(t,\xi) &= \; \int_U\E\left[ K^u_t (\xi^u)^2\rangle \right]\d u +\int_U\int_U\bar \xi^u\bar K_t(u,v)\bar \xi^v\d v \d u.
\end{align} 
The Riccati equation for $K^u$ is the following
\begin{align}
    \dot K_t^u + 2kK^u_t + \eta^u - (K^u_t)^2 = 0, \qquad K^u_T = r^u,
\end{align}
which can be explicitly solved as
\begin{align}
    K^u_t = \frac{-\eta^u\left(e^{(\delta^{u,+}-\delta^{u,-})(T-t)}-1\right)-r^u\left(\delta^{u,+}e^{(\delta^{u,+}-\delta^{u,-})(T-t)}-\delta^{u,-}\right)}{\left(\delta^{u,-}e^{(\delta^{u,+}-\delta^{u,-})(T-t)}-\delta^{u,+}\right)-r^u\left(e^{(\delta^{u,+}-\delta^{u,-})(T-t)}-1\right)},
\end{align}
where $\delta^{u,\pm} := -k \pm \sqrt{k^2 + \eta^u}$. 

\vspace{1mm}

The abstract Riccati equation for $\bar K(u,v)$ is written as 
\begin{align}
    &\dot{\bar K}_t(u,v) -k\tilde G_k(u,v)K^u_t - k\tilde G_k(u,v)K^v_t + 2k\bar K_t(u,v)\\
    &\quad-\int_U\bar K_t(u,w)k\tilde G_k(w,v)\d w - \int_U\tilde G_k(u,w)k\bar K_t(w,v)\d w\\
    &\quad +G_\eta(u,v) - \bar K_t(u,v)\left(K^u_t + K^v_t\right) \\
    &\quad -\int_U\bar K_t(u,w)\bar K_t(w,v)\d w = 0\\
    &\bar K_T(u,v) = G_r(u,v)
\end{align}

\vspace{1mm}

The linear equation for $Y$ is:
\begin{align}
    &\dot Y_t^u +kY^u_t -\int_Uk\tilde G_k(u,v)Y^v_t\d v -K^u_tY^u_t -\int_U\bar K_t(u,v)Y^v_t\d v =0,
\end{align}
with  $Y^u_T=0$; thus $Y^u_t=0$ for all $t\in[0,T]$, while  the linear equation for $\Lambda$ is
\begin{align}
    & (\sigma^u)^2K^u_t +\dot\Lambda^u_t =0,\qquad \Lambda^u_T=0,
\end{align}
which leads to $\Lambda_t^u = (\sigma^u)^2\int_t^TK^u_s\d s$. 

\vspace{1mm}

{Finally, the optimal control is
\begin{align}
    \hat \alpha^u_s = -K^u_sX^u_s -\int_U\bar K_s(u,v)\bar X^v_s\d v = -K^u_s\left(X^u_s + \int_U\frac{\bar K_s(u,v)}{K^u_s}\bar X^v_s\d v\right).
\end{align}}
\begin{Remark}
    If we set all the coefficients of the problem to be independent of $u$, and $\tilde G_k(u,v)=\tilde G_\eta(u,v)=\tilde G_r(u,v) =1$, we retrieve  the optimal control of the standard systemic risk model. In this case, we check that  $\bar K_t(u,v)=-K^u_t$ satisfies the abstract Riccati equation. Indeed, the abstract Riccati equation becomes:
    \begin{align}
        &\dot{\bar K}_t(u,v) -kK_t - kK_t + (k+k)\bar K_t(u,v) -\int_U\bar K_t(u,w)k\d w - \int_Uk\bar K_t(w,v)\d w\\
        &\quad +\eta - (\eta+\eta) - (K_t+K_t)\bar K_t(u,v) -\int_U\bar K_t(u,w)\bar K_t(w,v)\d w\\
        =& \; \dot{\bar K}_t(u,v) -2kK_t + 2k\bar K_t(u,v) -\int_U\bar K_t(u,w)k\d w - \int_Uk\bar K_t(w,v)\d w\\
        &\quad -\eta - 2K_t\bar K_t(u,v) -\int_U\bar K_t(u,w)\bar K_t(w,v)\d w  = 0, \\
        &\bar K_T(u,v) = -r.
    \end{align}
    It is easy to see that $\bar K_t(u,v):=-K_t$ satisfies the above equation (that actually reduces to the standard Riccati equation for $K$). Therefore the optimal control for the problem is $\hat\alpha_s = -K_s(X_s-\bar X_s)$.
\end{Remark}

{
\begin{Remark}
    The classical example of systemic risk (see \cite{pham-wei-bellman}) includes also mixed terms of the form $\alpha_sX_s$ and $\alpha_s\bar X_s$. In this heterogeneous framework one might consider a cost functional of the form
    \begin{small}
    \begin{align*}
    \int_U \E\Big[&
    \int_t^T\Big( 
    \eta^u\big(
    X^u_s -\int_U \tilde G_\eta(u,v)\,\bar X^{v}_s\,\d v
    \big)^2 + q^u\alpha^u_s\big(X^u_s-\int_UG_q(u,v)\bar X^v_s \d v\big) + (\alpha^u_s)^2\Big) \d s\\
    &+ r^u\big(
    X^u_T -\int_U \tilde G_r(u,v)\,\bar X^{v}_T\,\d v
    \big)^2 \Big]\, \d u,
    \end{align*}
    \end{small}
    where $q^u > 0$ is a positive parameter for the incentive to borrowing ($\alpha^u_t > 0$) or lending ($\alpha_t^u < 0$), and $G_q$ is possibly different from $\tilde G_\eta,\tilde G_r,\tilde G_k$.
    The computations carried out throughout this paper can be easily extended to the case including mixed terms. 
    In the systemic risk example, in particular, the optimal control would be given by
    \begin{align}
    \hat \alpha^u_s = -(K^u_s +\frac{q^u}{2})(X^u_s-\int_UG_q(u,v)\bar X^v_s\d v) -\int_U(\bar K_s(u,v) + K^u_sG_q(u,v))\bar X^v_s\d v,
    \end{align}
    where $\bar K$ is solution of the infinite-dimensional Riccati equation:
    \begin{align}\label{eq:barK_sysrisk}
    &\dot{\bar K}_t(u,v) -k\tilde G_k(u,v)K^u_t - k\tilde G_k(u,v)K^v_t + 2k\bar K_t(u,v)\\
    &\quad-\int_U\bar K_t(u,w)k\tilde G_k(w,v)\d w - \int_U\tilde G_k(u,w)k\bar K_t(w,v)\d w\\
    &\quad +G_\eta(u,v) - \bar K_t(u,v)\left(K^u_t + K^v_t + \frac{q^u+q^v}{2}\right) + K^u_t\frac{q^u}{2}G_q(u,v)\\
    &\quad + K^v_t\frac{q^v}{2}G_q(v,u) + \frac{(q^u)^2}{4}G_q(u,v) + \frac{(q^v)^2}{4}G_q(v,u)\\
    &\quad -\int_U\bigg( \bar K_t(u,w)\bar K_t(w,v) - \frac{q^w}{2}G_q(w,u)\bar K_t(w,v) \\
    &\qquad- \frac{q^w}{2}G_q(w,v)\bar K_t(u,w) + \frac{(q^w)^2}{4}G_q(w,u)G_q(w,v)\bigg)\d w = 0\\
    &\bar K_T(u,v) = G_r(u,v)
    \end{align}
    However, we refrain from the study of a more general model that includes mixed terms. The main challenge in this setting lies in establishing the existence and uniqueness of a solution $\bar K$ of equation \eqref{eq:barK_sysrisk}.
    Indeed, the proof of Proposition \ref{cor:a_priori_estimate_barK} does not apply to the case with mixed terms as the cost functional $J$ would no longer be positive. 
    We believe that, with additional effort and more refined techniques, it is possible to extend the a priori estimate in Proposition \ref{cor:a_priori_estimate_barK} to this general framework, but we leave it for further study.
\end{Remark}
}

\appendix

\section{Proof of Theorem \ref{exuniqsol}}\label{app:exuniqsol}

We define the map $\boldsymbol\Psi:L^2(U;C^d_{[t,T]})\to L^2(U;C^d_{[t,T]})$, $\boldsymbol\Psi(Y) = \E[X^Y]$, where $X^Y:u\mapsto X^{Y,u}$ and, for all $u\in U$, $X^{Y,u}$ is the solution of the standard SDE
\begin{align}\label{dyn_Y}
\begin{cases}
    &\d X^{Y,u}_s =\left [ \beta^u + A^u X^{Y,u}_s +  \int_U G_{  A}(u,v)Y^v_s \d v + B^u \alpha^u_s \right] \d s \\
    &\qquad\qquad +\left [\gamma^u + C^u  X^{Y,u}_s +    \int_U G_{C}(u,v) Y^v_s\d v + D^u \alpha^u_s \right ] \d W^u_s,\quad s\in[t,T]\\ 
    & X_t^{Y,u}=\xi^u.
\end{cases}
\end{align}
We will prove that $\boldsymbol{\Psi}$ has a unique fixed point $\bar Y$, which concludes the proof. 

\vspace{1mm}

\noindent \textit{Step 1.} We start by showing that the map $\boldsymbol{\Psi}$ is well defined. We fix $Y\in L^2(U,C^d_{[t,T]})$ and we consider the dynamics \eqref{dyn_Y} for $X^Y=(X^{Y,u})_u$. 
Standard results ensure existence and uniqueness of a solution $X^Y$ for a.e. $u$. From the standard Picard iteration, we have that for a.e. $u\in U$ fixed, 
\begin{equation}\label{eq:convL2}
\E\left[ \sup_{s\in[t,T]}|X^{Y,u}_s-X^{Y,n,u}_s|^2\right]\xrightarrow[]{n\to\infty}0,
\end{equation}
where $X^{Y,u,0}_s=\xi^u$ $\forall s$ and
\begin{align}
    X^{Y,u,n}_s &= \xi^u + \int_t^s[ \beta^u + A^u X^{Y,u,n-1}_r +  \int_U G_{  A}(u,v)Y^v_r \d v + B^u \alpha^u_r]\d r \\
    &\quad + \int_t^s [\gamma^u + C^u  X^{Y,u,n-1}_r +    \int_U G_{C}(u,v) Y^v_r\d v + D^u \alpha^u_r]\d W^u_r,\qquad\forall n.
\end{align} 
By induction, we have that $u\mapsto\E[X^{Y,u,n}_s]$ is a Borel measurable map for all $n,s$. 
{Indeed, for $n=0$, $u\mapsto \E[\xi^u]$ is measurable by hypothesis. For $n\ge 1$ we have that:
\begin{align}
    \E[X^{Y,u,n}_s] = \E\left[ \xi^u\right] + \int_t^s[ \beta^u + A^u \E\left[X^{Y,u,n-1}_r\right] +  \int_U G_{  A}(u,v)Y^v_r \d v + B^u \E\left[\alpha^u_r\right]\d r, 
\end{align}
so that the measurability of $\E\left[X^{Y,u,n-1}_s\right]$ implies the measurability of $\E\left[X^{Y,u,n}_s\right]$.
}
As for $u\mapsto\E[X^{Y,u,n,i}_{r}X^{Y,u,n,j}_{s}]$, we have that
\begin{align}
    \E[X^{Y,n,u,i}_{s}X^{Y,n,u,j}_{r}] =& \E\bigg[ \xi^{u,i}\left(\xi^{u,j} + \int_t^s\left[\beta^{u,j} + (A^uX^{Y,n-1,u}_q)^j +  (\int_U G_{A}(u,v)Y^v_q \d v)^j + (B^u \alpha^u_q)^j\right]\d q\right)\\
    &+ \xi^{u,j}\left(\xi^{u,i} + \int_t^r\left[\beta^{u,i} + (A^uX^{Y,n-1,u}_q)^i +  (\int_U G_{A}(u,v)Y^v_q \d v)^i + (B^u \alpha^u_q)^i\right]\d q\right)\\
    & \int_t^{s\wedge r}\left(\gamma^{u,i} + (C^u  X^{Y,u,n-1}_q)^i +    (\int_U G_{C}(u,v) Y^v_q\d v)^i + (D^u \alpha^u_q)^i\right)\\
    &\qquad\qquad\left(\gamma^{u,j} + (C^u  X^{Y,u,n-1}_q)^j +    (\int_U G_{C}(u,v) Y^v_q\d v)^j + (D^u \alpha^u_q)^j\right)\d q\bigg],
\end{align}
and {again} by induction and using the fact that $u\mapsto\E[\xi^{u,i}\alpha^{u,j}_{s}]$ is measurable for all $i,j,s$, that $u\mapsto\E[X^{Y,n,u,i}_{r}X^{Y,n,u,j}_{s}]$ is Borel measurable for all $i,j,n,r,s$.
Therefore, by \eqref{eq:convL2}, we have that $u\mapsto\E[X^{Y,u}_s]$ and $u\mapsto\E[X^{Y,u,i}_{r}X^{Y,u,j}_{s}]$ are Borel measurable maps for all $i,j,s,r$. {Finally we show that} $\E[X^Y]\in L^2(U;C^d_{[t,T]})$. 
{To prove this, we need to show that $\int_U\sup_{s\in[t,T]}|\E[X^{Y,u}_s]|^2\d u<\infty$, but we can actually show that the stronger condition $\int_U\E[\sup_{s\in[t,T]}|X^{Y,u}_s|^2]\d u<\infty$ holds.
Indeed, by using Gronwall inequality we obtain
\begin{align}
   & \int_U\E[\sup_{s\in[t,T]}|X^{Y,u}_s|^2]\d u \le M_T\int_U\left( \E[|\xi^u|^2] + \int_t^T \left( |\beta^u|^2 + \left|\int_U G_A(u,v)Y^v_r\d v\right|^2 +|B^u|^2\E[|\alpha_r^u|^2]\right)\d r  \right)\d u,
\end{align}
where $M_T$ is a positive constant.
To see that the right-hand side is finite, {using \eqref{eq:Hilbert-schmidt_integral}, we observe that $\int_t^T\int_U \left|\int_U G_A(u,v)Y^v_r\d v\right|^2\d u\d r \le \|G_A\|_{L^2}^2\int_t^T\int_U|Y^v_r|^2\d v\d r <\infty.$}
}

\vspace{1mm}

\noindent \textit{Step 2.} We will now prove that $\boldsymbol{\Psi}$ has a unique fixed point $\bar Y$, which concludes the proof.
Given $Y,Z\in L^2(U;C^d_{[t,T]})$, $r\in[t,T]$, we have by standard estimates 
{
\begin{align}
    \sup_{s\in[t,r]}\left|\E\left[X^{Y,u}_s-X^{Z,u}_s\right]\right|^2 \le M \int_t^r\left(\sup_{q\in[t,s]}\left|\E[X^{Y,u}_q-X^{Z,u}_q]\right|^2 + {\int_U |G_A(u,v)|^2 \d v} \int_U\sup_{q\in[t,s]}|Y^u_q-Z^u_q|^2\d u\right)\d s
\end{align}
}
where $M>0$ does not depend on $u$, thanks to the boundedness of the coefficients. Therefore, by Gronwall inequality 
\begin{equation}\label{sol_estim}
\sup_{s\in[t,r]}|\E[X^{Y,u}_s-X^{Z,u}_s]|^2 \le M{\int_U|G_A(u,v)|^2\d v}\int_t^r\int_U\sup_{q\in[t,s]}|Y^u_q-Z^u_q|^2\d u\d s.
\end{equation}
Therefore, taking $r=T$, we have that
\begin{align}
    \|\boldsymbol{\Psi}(Y) - \boldsymbol{\Psi}(Z)\|^2_{L^2(U;C^d_{[t,T]})} &= \int_U\sup_{s\in[t,T]}|\E[X^{Y,u}_s-X^{Z,u}_s]|^2\d u \le M{\|G_A\|_{L^2}^2}(T-t)\int_U\sup_{q\in[t,T]}|Y^u_q-Z^u_q|^2\d u\\
    &= M{\|G_A\|_{L^2}^2}(T-t)\|Y-Z\|^2_{L^2(U;C^d_{[t,T]})}.
\end{align}
By standard arguments, we can conclude that $\boldsymbol\Psi$ has a unique fixed point.

\section{Alternative formulations of the problem}\label{app:different-formulations}

We provide here two alternative formulations of the problem, that can easily be rewritten in the form of Section \ref{sec:formulation}.

\subsection{Centered formulation}\label{subsec:centered_formulation} 

Suppose you want to minimize, over all $\alpha \in \mathcal A$, a cost functional of the form
\begin{align*}
\tilde J_{}(t,\xi,\alpha):=\int_U \E\bigg[&
\int_t^T 
\left \langle Q^u \left(
X^u_s -\int_U \tilde G_Q(u,v)\,\bar X^{v}_s\,dv
 \right), X^u_s -\int_U \tilde G_Q(u,v)\,\bar X^{v}_s\,dv
\right \rangle + \langle R^u \alpha^u_s, \alpha^u_s \rangle \\
&+\left \langle H^u \left(
X^u_T -\int_U \tilde G_H(u,v)\,\bar X^{v}_T\,dv
 \right), X^u_T -\int_U \tilde G_H(u,v)\,\bar X^{v}_T\,dv
\right \rangle
\bigg]\,du, 
\end{align*}
where $Q, H \in L^\infty(U;\S_+^{d}),$ $ R\in L^\infty (U;\S_{>+}^{m}), \tilde G_{ Q},\tilde G_{ H}$ $\in$ ${L^2}(U\times U;\R^{d\times d})$, and  such that 
\begin{align}
\tilde G_{ Q}(u,v)=\tilde G_{ Q}(v,u)\trans, \quad\tilde G_{ H}(u,v)=\tilde G_{ H}(v,u)\trans.
\end{align} 
{Notice that $\tilde J(t,\xi,\alpha)\geq 0$}.
An example of a functional of this form  is described in Section \ref{sec:systemic_risk}  in the context of systemic risk models.

\vspace{1mm}

Let us check  that $\tilde J$ can be rewritten in the form \eqref{eq:J_theory}. 
We start with the following preliminary observations.
   \begin{Remark}\label{rem:properties_tildeG}
   \begin{enumerate}
       \item The following equality holds (recall that $Q^u, H^u$ are symmetric)
       \begin{small}
\begin{align*}
   & \int_U \left\langle Q^u\int_U \tilde G_Q(u,v)\,y^{v} \,dv , \int_U \tilde G_Q(u,v)\,y^{v} \,dv \right\rangle \d u\\
   & = \int_U  \int_U  \int_U \left\langle Q^u  \tilde G_Q(u,v)y^{v}, \tilde G_Q(u,w) y^{w}\right\rangle \d v   \d w \d u \\
   &=\int_U \int_U  \int_U \left\langle Q^w  \tilde G_Q(w,v)y^{v}, \tilde G_Q(w,u) \bar X^{u}_s\right\rangle \d v  \d u \d w \\
   &=\int_U \left\langle y^{u}, \int_U   \int_U \tilde G_Q (w,u)^TQ^w  \tilde G_Q (w,v) \d w y^{v} \d v\right\rangle \d u\\
   &=\int_U \left\langle y^{u}, \int_U   \int_U \tilde G_Q (u,w)Q^w  \tilde G_Q (w,v) \d w y^{v} \d v\right\rangle \d u,
\end{align*}
 \end{small}
(and an analogous equality holds for the term in $H^u$).
\item  We observe that 
\begin{align}\label{eq:symmetric_structure_graphon_cost_term}
    \int_U \langle y^u, Q^u\int_U \tilde G_Q(u,v)y^v\d v \rangle \d u = \int_U \langle y^u, \int_U\tilde G_Q(u,v)\frac{Q^u + Q^v}{2} y^v \d v \rangle \d u,
\end{align}
(and similarly for the term in $H^u$).
Indeed, we have (recall that $Q$ is symmetric)
\begin{align}
& \int_U \langle y^u, Q^u\int_U \tilde G_Q(u,v) y^v\d v \rangle \d u=\int_U\int_U \langle y^u,Q^u\tilde G_Q(u,v)y^v\rangle\d v\d u \\
&= \int_U\int_U \langle y^v,Q^v\tilde G_Q(v,u)y^u\rangle\d v\d u=\int_U \langle y^u, \int_U \tilde G_Q(u,v) Q^v y^v \d v \rangle \d u,
\end{align}
so that, summing $\int_U \langle y^u, Q^u\int_U \tilde G_Q(u,v)y^v\d v \rangle \d u $ on both members of the previous equality, we have the claim.
   \item Using the previous points, we have 
\begin{align}
{0}&{\leq} \int_U \E\left[\left \langle Q^u \left(
X^u -\int_U \tilde G_Q(u,v)\,\bar X^{v}\,dv
 \right),  X^u -\int_U \tilde G_Q(u,v)\, \bar X^{v}\,dv
\right \rangle\right]\\
&= \int_U \E[\left\langle Q^u X^u, X^u\right\rangle] \d u + \int_U \left\langle \bar X^{u}, \int_U  \left( \int_U \tilde G_Q(u,w)Q^w  \tilde G_Q(w,v) \d w - 2Q^u  \tilde G_Q(u,v) \right) \bar X^{v}  \d v\right\rangle \d u\\
&=\int_U \E[\langle  X^u,Q^u X^u\rangle] du + \int_U \langle \bar X^u,  \int_UG_{Q}(u,v)\bar X^v \d v \rangle  du,
\end{align}
where we have set $G_Q(u,v):=\int_U \tilde G_Q(u,w)Q^w \tilde G_Q(w,v) \d w -  (Q^u+Q^v)  \tilde G_Q(u,v)$. An analogous property holds for the term in $H$, by setting  $G_H(u,v):=\int_U \tilde G_H(u,w)H^w  \tilde G_H(w,v) \d w - ( H^u +H^v) \tilde G_H(u,v)$. {Hence, condition \eqref{eq:positivity_hpII} is satisfied.}
   \end{enumerate}
\end{Remark}

\vspace{1mm}

Developing the square in $\tilde J$ and using the previous remark,  we have
\begin{align*}
\tilde J_{}(t,\xi,\alpha) & = \;  \int_U \E\Big[
\int_t^T  \Big( \left\langle Q^u X^u_s, X^u_s\right\rangle + \left\langle\bar X^{u}_s, \int_U G_Q(u,v)  \bar X^{v}_s  \d v\right\rangle \\
& \qquad \quad \quad   + \left\langle R^u\alpha^u_s,\alpha^u_s\right\rangle\Big) \d s   +\left\langle H^uX^u_T,X^u_T\right\rangle \Big]\,\d u+ \left\langle\bar X^{u}_T, \int_U G_H(u,v) \bar X^{v}_T  \d v\right\rangle.
\end{align*}

\subsection{Symmetric formulation}\label{subsec:symmetric_formulation}

Suppose you want to minimize, over all $\alpha \in \mathcal A$, a cost functional of the form
\begin{align*}
    \tilde J_{}(t,\xi,\alpha) := \int_U\E\Big[ & \int_t^T \langle X^u_s,Q^u X^u_s \rangle + \langle  \int_U \tilde G_{Q}(u,v)\bar X^v_s\d v,\bar Q^u\int_U \tilde G_{Q}(u,v)\bar X^v_s\d v\rangle + \langle \alpha^u_s,R^u \alpha^u_s \rangle\\
    &+ \langle X^u_T, H^u X^u_T \rangle + \langle \int_U \tilde G_{H}(u,v)\bar X^v_T\d v ,\bar H^u \int_U \tilde G_{H}(u,v)\bar X^v_T\d v \rangle \Big]\d u,
\end{align*}
where $Q, H \in L^\infty(U;\S_+^{d}),$ $ R\in L^\infty (U;\S_{>+}^{m})$, $\tilde G_{ Q}, \tilde G_{ H}$ $\in$ ${L^2}(U\times U;\R^{d\times d})$.

{In a similar way} to the centered case, using Remark \ref{rem:properties_tildeG}, $\tilde J$    can be rewritten into the  form \eqref{eq:J_theory} with 
$G_Q(u,v):=\int_U \tilde G_Q(u,w)Q^w \tilde G_Q(w,v) \d w $, $G_H(u,v):=\int_U \tilde G_H(u,w)H^w  \tilde G_H(w,v) \d w $.

\begin{small}
\paragraph{\textbf{Acknowledgments.}} The authors would like to thank Samy Mekkaoui and the two anonymous referees for their careful reading of the paper and useful remarks. Filippo de Feo is grateful to Huyên Pham for the invitation to Universit\'e Paris Cit\'e to work on this paper. Filippo de Feo acknowledges funding by the Deutsche Forschungsgemeinschaft (DFG, German Research Foundation) – CRC/TRR 388 "Rough Analysis, Stochastic Dynamics and Related Fields" – Project ID 516748464, by INdAM (Instituto Nazionale di Alta Matematica F. Severi) - GNAMPA (Gruppo Nazionale per l'Analisi Matematica, la Probabilità e le loro Applicazioni), and by the Italian Ministry of University and Research (MUR), in the framework of PRIN projects 2017FKHBA8 001 (The Time-Space Evolution of Economic Activities: Mathematical Models and Empirical Applications) and 20223PNJ8K (Impact of the Human Activities on the Environment and Economic Decision Making in a Heterogeneous Setting: Mathematical Models and Policy Implications).
\end{small}

\bibliographystyle{plain}
\bibliography{biblio}

\end{document}